\documentclass[reqno,  11pt, epsfig, amsfonts]{amsart}
\usepackage{amssymb,  amsmath,  amsthm}
\usepackage[colorlinks=true, linkcolor=blue, citecolor=red]{hyperref}
\usepackage[nameinlink,capitalize,noabbrev]{cleveref}%智能引用
\usepackage{color}
\usepackage{graphicx}

\usepackage{mathrsfs}%花体
\usepackage[T1]{fontenc}

\usepackage{upgreek}
\usepackage{tcolorbox}

\usepackage{commath}
\usepackage{graphicx}
\usepackage{fancyhdr}%页眉页脚
\usepackage{framed}
\usepackage{exscale}%增大积分号
\usepackage{relsize}
\usepackage{array}%制表
\usepackage{tcolorbox}%文字加框
\usepackage {subcaption}%图片加注释
\usepackage{lscape}%设置纸张方向
\usepackage{float}%强制插图
\usepackage[utf8]{inputenc}
\usepackage{caption}

\usepackage{tikz}
\tikzset{
	inputnode/.style={draw, circle, inner sep=2pt, font=\scriptsize},
	hiddennode/.style={draw, circle, inner sep=2pt, font=\scriptsize},
	outputnode/.style={draw, circle, inner sep=2pt, font=\scriptsize},
}

\newtheorem{theorem}{Theorem}[section]
\newtheorem{lemma}[theorem]{Lemma}

\newtheorem{cor}[theorem]{Corollary}

\theoremstyle{definition}
\newtheorem{definition}[theorem]{Definition}

\newtheorem{Hypothesis}[theorem]{Hypothesis}	

\newtheorem{notation}[theorem]{Notation}

\theoremstyle{remark}
\newtheorem{remark}[theorem]{Remark}

\numberwithin{equation}{section}

\theoremstyle{remark}

\newtheorem*{Outline*}{Outline}

\numberwithin{equation}{section}

\usepackage{yhmath}%弧
\usepackage{enumitem}%\item

\title[Topological Horseshoe Induced by Periodic Switching]{Topological Horseshoe Induced by Periodic Switching Between Non-Isochronous Planar Systems}

\author[J. Cheng]{Junfeng Cheng}
\address{School of Mathematics and Statistics,  Huazhong University of Science and Technology,  Wuhan, 430074, China}

\email{chengjf@hust.edu.cn}
\author[X.-S. Yang]{Xiao-Song Yang}
\address{School of Mathematics and Statistics,  Huazhong University of Science and Technology,  Wuhan, 430074, China}
\address{Hubei Key Laboratory of Engineering Modeling and Scientific Computing,  Huazhong University of Science and Technology,  Wuhan 430074, China}
\email{yangxs@hust.edu.cn}
\begin{document}

%    \subjclass is required.

%37Dxx Dynamical systems with hyperbolic behavior
%
%37D45 Strange attractors, chaotic dynamics of systems with hyperbolic behavior
%
%37Jxx Dynamical aspects of finite-dimensional Hamiltonian and Lagrangian systems
%
%37J46 Periodic, homoclinic and heteroclinic orbits of finite-dimensional Hamiltonian systems

%
%34Cxx Qualitative theory for ordinary differential equations [See also 37-XX]
%
%34C25 Periodic solutions to ordinary differential equations
%34C28 Complex behavior and chaotic systems of ordinary differential equations [See also 37Dxx]

\date{}

\dedicatory{}


\begin{abstract}
	We establish a criterion for the existence of a topological horseshoe in a class of planar systems generated by periodic switching between two subsystems, each admitting a family of closed orbits, where the mechanism for chaos arises from the non-isochronicity of each subsystem.
	
	Exploiting the relationship between the period function of a Hamiltonian system and the rate of change of the area enclosed by its periodic orbits, we derive a criterion, which can be checked by numerical methods, for the existence of horseshoe in planar systems obtained by switching between two Hamiltonian subsystems. Furthermore, by invoking monotonicity results for the period function in Newtonian Hamiltonian systems, we obtain an explicit and computable criterion that guarantees chaotic dynamics in planar systems generated by switching between two such subsystems.
\end{abstract}

\maketitle

%    Text of article.
\section{Introduction}
Switching systems constitute an important class of dynamical systems characterized by a family of continuous-time subsystems and a rule that governs the switching between them.
More precisely, such systems can be described in the form
\begin{equation}
	\dot{x}(t) = f_{\sigma}(x(t), t),
\end{equation}
where \(x(t)\in\mathbb{R}^n\) denotes the state variable, \(\{f_i\}\) is a collection of vector fields, and \(\sigma\) is a piecewise constant function of time, called a switching signal \cite{Liberzon1999}.

In specific situations, the value of $\sigma$ at a given time $t$ may depend solely on $t$ or $x(t)$, or both, or may be generated using more sophisticated techniques such as hybrid feedback with memory in the loop\cite{Liberzon1999}. 
Switching signals that depend explicitly on the state are referred to as state-dependent switching, whereas those determined solely by time are called periodic switching. It is well known that state-dependent switching mechanisms can induce complex dynamical behaviors, including chaos, even when each individual subsystem is relatively simple. For instance, it was verified in Ref.~\cite{Yang2002,YangLi2005} that the switching circuit
\begin{eqnarray}\label{ss}
	\begin{pmatrix}
		\overset{.}{x}\\ \overset{.}{y}\\ \overset{.}{z}
	\end{pmatrix}=\begin{pmatrix}
		r-1 & -1 & 0\\
		r   & -1 & 0\\
		0   & 0  & -1
	\end{pmatrix}\begin{pmatrix}
		x\\y\\z-c
	\end{pmatrix},
\end{eqnarray}
with parameters
\begin{equation}
	\begin{aligned}
		r&=2.1,\quad c=0,\quad \text{for } 2x+z\leq 1,\\
		r&=1,\quad c=0.9,\quad \text{otherwise},
	\end{aligned}
\end{equation}
exhibits chaotic dynamics via topological horseshoe theory.

The study of topological horseshoe theory was initiated by the proposal of the concept of the Smale horseshoe \cite{Smale1967}.  After that, the notion of the topological horseshoe was introduced by Kennedy and Yorke \cite{Kennedy2001-1,Kennedy2001-2}. Subsequently, Yang and Tang enhanced this concept in Ref.~\cite{Yang2004} to deal with piecewise continuous maps. 
These developments provide a powerful framework for detecting chaotic dynamics. 
In many existing works, numerical methods, such as Runge-Kutta schemes, are employed to investigate such behavior (see, e.g., \cite{Cheng2024,Cheng2025b}). 

This paper adopts a purely analytical approach to establish the existence of topological horseshoes in periodic switching systems.
More specifically, we identify a novel mechanism for the occurrence of chaos in a class of periodic switching systems, arising from the non-isochronicity within each subsystem. Motivated by the work of Barrientos et al.~\cite{Barrientos2017}, where chaotic dynamics were rigorously established in a limit case of a forced SIR epidemic model, we derive a criterion that combines both geometric and analytical conditions for the emergence of chaotic dynamics in planar periodic switching systems.
More precisely, under suitable assumptions on the families of closed orbits of the two subsystems, we establish the existence of a topological horseshoe by exploiting the strict monotonicity of the associated period functions.

The monotonicity of the period function is, in general, difficult to verify. We therefore restrict our attention to Hamiltonian systems. In this setting, we derive a computable criterion based on the relationship between the area enclosed by periodic orbits and the associated period function. Furthermore, for Newtonian Hamiltonian systems, where the monotonicity of the period function has been extensively studied (see, e.g., \cite{Chouikha1999}), we invoke classical results in Ref.~\cite{Chow1986} to obtain sufficient conditions for the monotonicity of the period function in such switching systems.

Finally, we present two illustrative examples. The first is a toy example demonstrating the applicability of our criteria to a simple periodic switching system. The second result improves upon the existing results in Ref.~\cite{Barrientos2017} by showing that the lower bounds on the dwell times of the subsystems are not required to be symmetric.

The remainder of this paper is organized as follows. In \cref{preliminary}, we briefly review the notion of topological horseshoes and the relevant theoretical framework. In \cref{general_systems}, we establish our main results for planar periodic switching systems.
In \cref{Hamiltonian_systems}, we further study the systems switching between two Hamiltonian subsystems, and, by exploiting properties of the associated period functions, derive corresponding criteria. In \cref{application}, we present two illustrative examples.
Finally, in \cref{discussion}, we compare our results with those in Ref.~\cite{Barrientos2017} and highlight several open problems that require further investigation, and in \cref{summary}, we conclude the paper with a summary of the main results and contributions.

\section{Review of topological horseshoes}\label{preliminary}
In this section, we present a review of topological horseshoes for the purposes of this paper.
Prior to exploring the concept of topological horseshoes, it is crucial to introduce the definition of semi-conjugation. The subsequent two definitions are referenced from Ref.~\cite{Wiggins1988}.

\begin{definition}\cite{Wiggins1988}
	Let $M$ and $N$ be topological spaces and consider two continuous maps $f:M\to M$ and $g:N\to N$. The map $f$ is said to be semi-conjugate to $g$ if there is a continuous surjective map $h:M\to N $  such that 
	\begin{eqnarray}
		h\circ f=g\circ h.
	\end{eqnarray}
\end{definition}

\begin{definition}\cite{Wiggins1988}
	Let $X$ be a metric space, and let $\mathcal{A}=\{0,1,2,\dots,m-1\}$ be a finite alphabet with $m$ symbols. 
	The $m$-shift map $\sigma$ is defined as 
	\begin{eqnarray}
		\sigma(s)_i=s_{i+1},
	\end{eqnarray}
	where $s\in\mathcal{A}^\mathbb{Z}=\{x=(x_i)_{i\in\mathbb{Z}}:x_i\in\mathcal{A}~\text{for all } i\in\mathbb{Z}\}$.
	
	Consider a  (piecewise) continuous map $f:X\to X$. If there exists a compact invariant set $\Lambda \subset X $ such that the restriction of $f$ to $\Lambda$ is semi-conjugate to the m-shift map $\sigma $,
	then $f$ is said to have an m-type topological horseshoe.  
\end{definition}

For the practical identification of horseshoes within applied problems, we also need to introduce the concept of crossing, detailed in Ref.~\cite{Yang2009}. Let us consider a compact and connected region \(D \subset \mathbb{R}^n\). Let \(B_i\) (\(i = 1, 2, \ldots, m\)) be compact, path-connected subsets  in \(D\), each homeomorphic to the unit cube. Denote the boundary of each set \(B_i\) by \(\partial B_i\), and consider a piecewise continuous map $f: D \rightarrow X$, which is continuous on each of the compact sets $B_i$.

\begin{definition} \cite{Yang2009}
	For each $B_i, 1\leq i\leq m$, let $B_i^1$ and $B_i^2$ be two fixed disjoint connected nonempty compact subsets (usually pieces of $\partial B_i$) contained in the boundary $\partial B_i$. A connected subset $l$ of $B_i$ is said to be a connection of $B_i^1$ and $B_i^2$ if $l\cap B_i^1\neq \emptyset$ and $l\cap B_i^2\neq \emptyset$. 
\end{definition}

\begin{definition}\label{thd} \cite{Yang2009}
	Let $l\subset B_i$ be a connection of $B_i^1$ and $B_i^2$. We say that $f(l)$ is crossing $B_j$, if $l$ contains a connected subset $\bar{l}$ such that $f(\bar{l})$ is a connection of $B_j^1$ and $B_j^2$, i.e., $f(\bar{l})\subset B_j$, while $f(\bar{l})\cap B_j^1\neq \emptyset$ and $f(\bar{l})\cap B_j^2\neq \emptyset$. In this case we denote it by $f(l)\mapsto B_j$. Furthermore, if $f(l) \mapsto B_j $ for every connection $l$ of $B_i^1$ and $B_i^2$, then $f(B_i)$ is said to be crossing $B_j$ and denoted by $f(B_i)\mapsto B_j$. To simplify the terminology, we refer to $B_i$ as a crossing block of $B_j$, and for clarity, we call this crossing the dimension one crossing.
\end{definition}

In light of the aforementioned definitions, we recall the following lemma.
\begin{lemma}\label{thm1} \cite{Yang2004} 
	Suppose that the map $f:D\to \mathbb{R}^n$ satisfies the following assumptions:\\		
	\noindent(1) There exist $m$ mutually path-connected disjoint compact subsets $B_1,B_2,\dots$ and $B_m$ of $D$, the restriction of $f$ to each $B_i$, i.e., $f|_{B_i}$ is continuous.\\		
	\noindent (2) The dimension one crossing relation $f(B_i)\mapsto B_j $ holds for $1\leq i,j\leq m$.\\	
	Then there exists a compact invariant set $K\subset D$, such that $f|_K$ is semi-conjugate to a m-shift map.
\end{lemma}

To describe the chaos of a system quantitatively, we recall a lemma on the topological entropy. For more details, readers are referred to Ref.~\cite{Robinson1995}.

\begin{lemma}\label{TE} \cite{Robinson1995}
	Let $X$ be a compact space, and $f:X\to X$ a continuous map. If there exists an invariant set $\Lambda\subset X$ such that $f|_\Lambda$ is semi-conjugate to the m-shift map $\sigma$, then we have
	\begin{equation}
		h(f)\geq h(\sigma)=\log m,
	\end{equation}
	where $h(f)$ denotes the topological entropy of the map $f$. 
	%	Furthermore, for every positive integer $k$, we have the following fact
	%	\begin{eqnarray}
		%		h(f^k)=kh(f).
		%	\end{eqnarray}
\end{lemma}

\section{Topological horseshoe induced by periodic switching in planar systems with a family of closed orbits}\label{general_systems}

Consider the $T$-periodic switching system in $\mathbb{R}^2$ composed of two  subsystems,
\begin{equation}\label{SwitchingSystem}
	\dot{x}=
	\begin{cases}
		f_1(x), \qquad nT &\leq t< nT+T_1,\\
		f_2(x), \quad nT+T_1&\leq t<nT+T_1+T_2,
	\end{cases}
\end{equation}
where $T_i$ denotes the dwell time associated with the vector field $f_i(x)$, and $T=T_1+T_2$.

The following hypothesis characterizes the necessary conditions that the family of closed orbits associated with each subsystem must satisfy.
\begin{Hypothesis}\label{hypothesis1}
	For each $i=1,2$, the subsystem $\dot{x}=f_i(x)$ has a family of non-trivial periodic orbits, and in particular there exist two closed orbits $\gamma_i$ and $\Gamma_i$ such that $\gamma_i$ is contained in the interior of the region enclosed by $\Gamma_i$.
	Let $\mathcal{R}_i$ denote the open annular region bounded by $\gamma_i$ and $\Gamma_i$. 
\end{Hypothesis}

The following hypothesis guarantees that, for each fixed \( t \), the associated flow map \( \phi_i(t,\cdot) \) defines a self-homeomorphism on \( \mathcal{R}_i \).
\begin{Hypothesis}\label{hypothesis2}
	Each vector field \( f_i \) is locally Lipschitz continuous on \( \mathcal{R}_i \).
\end{Hypothesis}

To characterize the period of closed orbits contained in $\mathcal{R}_i$, we introduce the following notation.
\begin{notation}
	Denote by $s_i$ and $S_i$ the areas enclosed by $\gamma_i$ and $\Gamma_i$, respectively. Then, for fixed $i$ and each $c\in[0,1]$, there exists a unique periodic orbit $\gamma$ whose enclosed area is
	\begin{equation}
		s_i + c(S_i - s_i).
	\end{equation}
	This defines a bijection between $c\in[0,1]$ and periodic orbits contained in $\mathcal{R}_i$. Let $\mathcal{E}_i$ denote the inverse of this correspondence. Accordingly, the period function
	\begin{equation}
		T_i:[0,1]\to \mathbb{R}^+,
	\end{equation}
	which assigns to each $c$ the minimal period of the corresponding orbit $\gamma=\mathcal{E}_i^{-1}(c)$, is well defined.
\end{notation}

Under this notation, we have
\begin{equation}
	\mathcal{R}_i=\{\gamma\subset \mathbb{R}^2:\ 0<\mathcal{E}_i(\gamma)<1\}, \quad i=1,2.
\end{equation}

Throughout this paper, the time-\(T\) Poincar\'e map is defined in the same way as the context of the Melnikov method \cite{Guckenheimer1984,Wiggins2013}. More precisely, we introduce the following definition.
\begin{definition}
	The $T$-Poincar\'e map associated with \eqref{SwitchingSystem} is defined by
	\begin{equation}\label{poincare_map1}
		P(x) = \phi_j(T_j, \phi_i(T_i, x)),
	\end{equation}
	where \( \phi_i(t,\cdot) \) denotes the flow generated by \( \dot{x} = f_i(x) \).
\end{definition}

We are now in a position to state the main result of this paper.
\begin{theorem}\label{MainTheorem1}
	Assume that \cref{hypothesis1} and \cref{hypothesis2} hold.
	Suppose further that the following conditions are satisfied:	
	\begin{enumerate}[label=(\roman*)]
		\item There exist indices $i\neq j$ such that either
		$\gamma_i\cap\mathcal{R}_j\neq\emptyset$ and
		$\gamma_i\not\subset\mathcal{R}_j$, or
		$\Gamma_i\cap\mathcal{R}_j\neq\emptyset$ and
		$\Gamma_i\not\subset\mathcal{R}_j$;
		\item The period function $T_i(c)$ is continuous and strictly monotone on $(0,1)$.
	\end{enumerate}
	Then, there exist constants $T_1^*>0$ and $T_2^*>0$ such that, whenever
	\begin{equation}
		T_i \ge 5T_i^*, \quad T_j \ge 3T_j^*, \quad i,j\in{1,2},\ i\neq j,
	\end{equation}
	the time-$T$ Poincar\'e map \eqref{poincare_map1}
	exhibits a topological horseshoe, where $T=T_1+T_2$. 
	%	More precisely, the Poincar\'e map is defined by
	%	\begin{equation}\label{poincare_map1}
		%		P(x)=\phi_j(T_j,\phi_i(T_i,x)).
		%	\end{equation}
	%	where $\phi_i(t,\cdot)$ denotes the flow generated by $\dot{x}=f_i(x)$.
\end{theorem}

The following corollary is an immediate consequence of \cref{TE} and \cref{MainTheorem1}.
\begin{cor}
	Under the assumptions of \cref{MainTheorem1}, the Poincar\'e map $P$ defined in \eqref{poincare_map1} satisfies the following lower bound on its topological entropy:
	\begin{equation}
		h(P)\geq \log 2,
	\end{equation}
	where $h(P)$ denotes the topological entropy of $P$.
\end{cor}

Before presenting the proof of \cref{MainTheorem1}, we introduce a lemma that clarifies the geometric meaning of Condition~(i).
\begin{lemma}\label{lem1}
	Assume that \cref{hypothesis1} holds. 
	If Condition~(i) of \cref{MainTheorem1} is satisfied,
	then there exists an open set $\mathcal{Q}\subset \mathcal{R}_1\cap\mathcal{R}_2$ whose boundary consists of four arcs belonging to closed orbits of the subsystems, such that $\mathcal{Q}$ is homeomorphic to a quadrilateral. Moreover, the preimages of any two adjacent edges of this quadrilateral lie on closed orbits of different subsystems.
\end{lemma}

\begin{proof}[Proof of \cref{lem1}] 
	For convenience, we fix $i=1$ and $j=2$, and consider the following two situations:
	\begin{equation}
		\gamma_1\cap\mathcal{R}_2\neq\emptyset \quad \text{and} \quad \gamma_1\not\subset \mathcal{R}_2,
	\end{equation}
	\begin{equation}
		\Gamma_1\cap\mathcal{R}_2\neq\emptyset \quad \text{and} \quad \Gamma_1\not\subset \mathcal{R}_2.
	\end{equation}
	Let $\mathcal{I}_i$ denote the region enclosed by $\gamma_i$, and let $\mathcal{O}_i$ denote the unbounded region exterior to $\Gamma_i$. Then the plane has the decomposition
	\begin{equation}
		\mathbb{R}^2=\mathcal{I}_i\sqcup\gamma_i\sqcup\mathcal{R}_i\sqcup\Gamma_i\sqcup\mathcal{O}_i,\qquad i=1,2.
	\end{equation}
	
	\medskip
	\noindent\textbf{Case I.} $\gamma_1\cap\mathcal{R}_2\neq\emptyset$ and $\gamma_1\not\subset \mathcal{R}_2$.
	
	\medskip
	Since $\gamma_1\not\subset \mathcal{R}_2$, there exists a point on $\gamma_1$ that lies outside $\mathcal{R}_2$. Consequently, at least one of the following two subcases must occur:
	\begin{equation}
		\gamma_1\cap\mathcal{O}_2\neq\emptyset,
	\end{equation}
	or
	\begin{equation}
		\gamma_1\cap\mathcal{I}_2\neq\emptyset.
	\end{equation}

	\noindent\textbf{Subcase I.} $\gamma_1\cap\mathcal{O}_2\neq\emptyset$.
	
	\medskip
	In this subcase, there exist two points $A$ and $B$ such that
	\begin{equation}
		A\in\gamma_1\cap\mathcal{R}_2,\qquad 
		B\in\gamma_1\cap\mathcal{O}_2.
	\end{equation}
	Let $l:[0,1]\to\gamma_1$ be a continuous parametrization of the arc $\wideparen{AB}\subset\gamma_1$ connecting $A$ to $B$. Define
	\begin{equation}
		t_*=\inf\bigl\{t\in[0,1]: l(s)\in\Gamma_2\cup\mathcal{O}_2 \text{ for all } s\in[t,1]\bigr\}.
	\end{equation}
	Set the point $C=l(t_*)$. By construction, we have $C\in\gamma_1\cap\Gamma_2$, and hence
	\begin{equation}
		\mathcal{E}_1(C)=0,\qquad \mathcal{E}_2(C)=1.
	\end{equation}
	
	%	Since each energy level in a neighborhood of $C$ uniquely determines a closed orbit, the inverse mappings $\mathcal{E}_i^{-1}(\cdot)$, $i=1,2$, are well defined locally.
	It follows that there exist $\varepsilon_1,\varepsilon_2>0$ such that the closed orbits
	\begin{equation}
		\tilde{\gamma}_1:=\mathcal{E}_1^{-1}(\varepsilon_1),
		\qquad
		\tilde{\Gamma}_2:=\mathcal{E}_2^{-1}(1-\varepsilon_2)
	\end{equation}
	are well defined. The curves $\gamma_1$, $\Gamma_2$, $\tilde{\gamma}_1$, and $\tilde{\Gamma}_2$ bound a region $\mathcal{Q}$ that is homeomorphic to a quadrilateral, with $C$ as one of its vertices. Moreover,
	\begin{equation}
		\mathcal{Q}\subset\mathcal{R}_1\cap\mathcal{R}_2.
	\end{equation}
	
	For the reader’s convenience, an illustration  is provided in \cref{illustration1}.
	\begin{figure}
		\centering
		\includegraphics[width=0.8\textwidth]{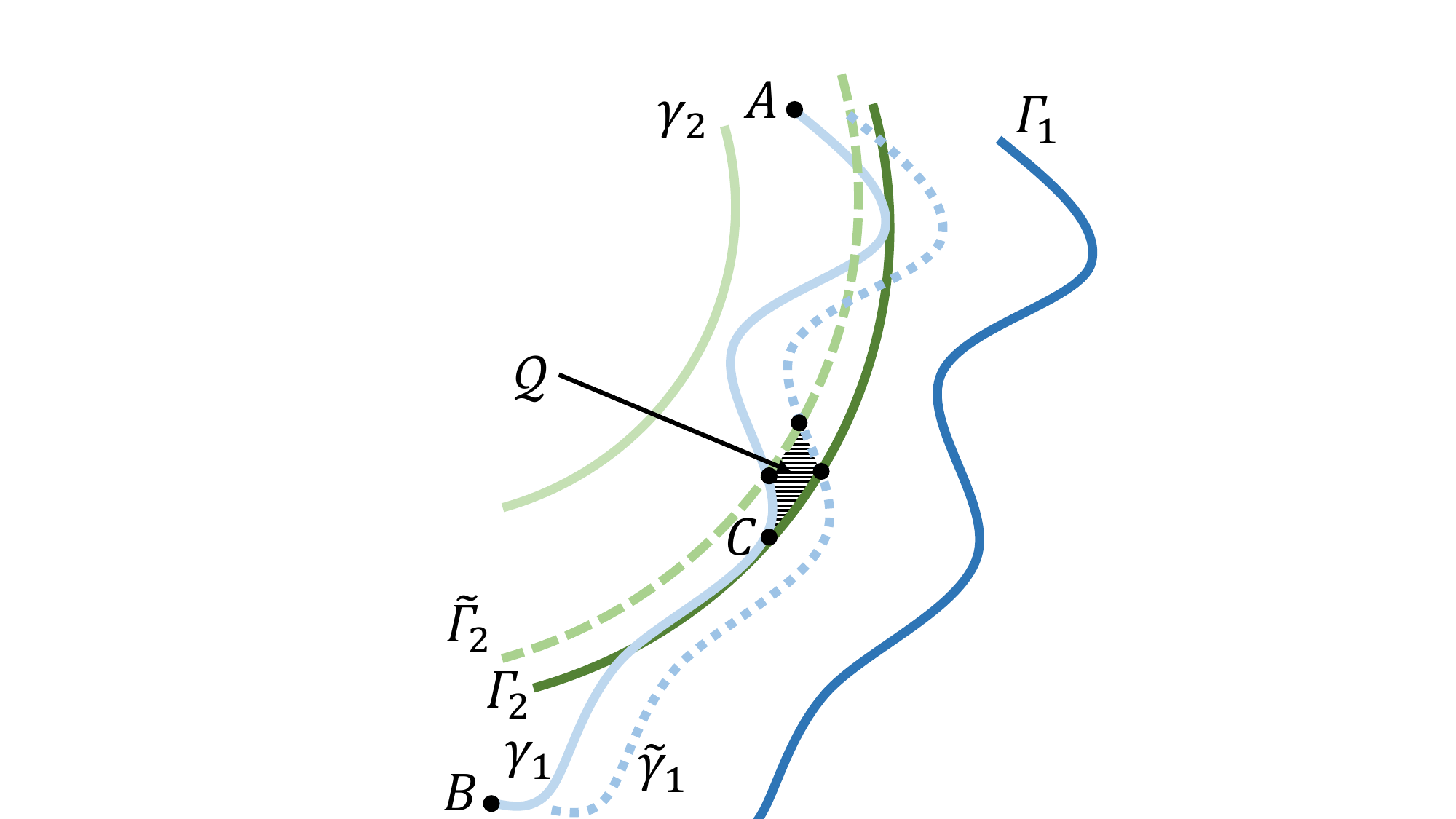}
		\caption{Case I, Subcase I: $\gamma_1\cap\mathcal{O}_2\neq\emptyset$. Solid curves denote portions of the orbits $\gamma_i,\Gamma_i$ $(i=1,2)$, whereas dotted curves and dashed curves denote portions of the orbits $\tilde{\gamma}_1$ and $\tilde{\Gamma}_2$, respectively.
		}
		\label{illustration1}
	\end{figure}
	
	\noindent\textbf{Subcase II.} $\gamma_1\cap\mathcal{I}_2\neq\emptyset$.
	
	\medskip
	Suppose there exist points $A\in \mathcal{R}_2$ and $B\in \mathcal{I}_2$, and let $l:[0,1]\to \mathbb{R}^2$ be a continuous path connecting $A$ to $B$, with $l(0)=A$ and $l(1)=B$. Define
	\begin{equation}
		t_*=\inf\bigl\{t\in[0,1]: l(s)\in \gamma_2\cup \mathcal{I}_2 \text{ for all } s\in [t,1] \bigr\},
	\end{equation}
	and set $C=l(t_*)$.
	
	There exist $\varepsilon_1,\varepsilon_2>0$ such that the following two orbits
	\begin{equation}
		\tilde{\gamma}_1:=\mathcal{E}_1^{-1}(\varepsilon_1), \qquad
		\tilde{\gamma}_2:=\mathcal{E}_2^{-1}(\varepsilon_2)
	\end{equation}
	are well defined. Then $\gamma_1$, $\gamma_2$, $\tilde{\gamma}_1$, and $\tilde{\gamma}_2$ bound a region $\mathcal{Q}$ homeomorphic to a quadrilateral, with $C$ as one of its vertices. Clearly,
	\begin{equation}
		\mathcal{Q}\subset \mathcal{R}_1\cap \mathcal{R}_2.
	\end{equation}
	An illustration is provided in \cref{illustration2}.
	\begin{figure}
		\centering
		\includegraphics[width=0.8\textwidth]{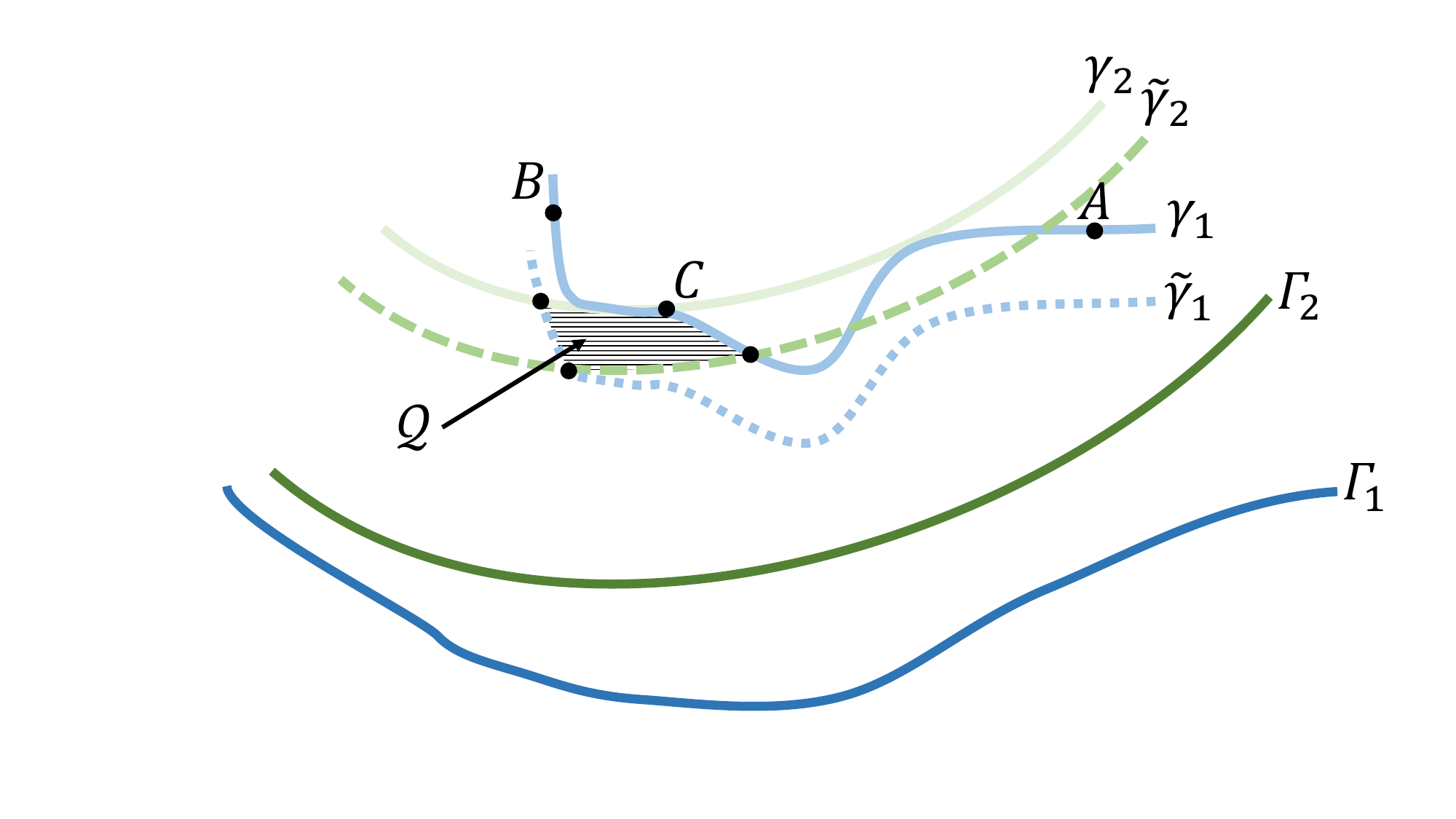}
		\caption{Case I, Subcase II: $\gamma_1\cap\mathcal{I}_2\neq\emptyset$. Solid curves denote portions of the orbits $\gamma_i,\Gamma_i$ $(i=1,2)$, whereas dotted curves and dashed curves denote portions of the orbits $\tilde{\gamma}_1$ and $\tilde{\gamma}_2$, respectively.}
		\label{illustration2}
	\end{figure}
	
	\medskip
	\noindent\textbf{Case II.} $\Gamma_1\cap \mathcal{R}_2 \neq \emptyset$ and $\Gamma_1 \not\subset \mathcal{R}_2$.
	
	\medskip
	\noindent\textbf{Subcase I.} $\Gamma_1\cap \Gamma_2 \neq \emptyset$.
	
	\medskip
	There exist $\varepsilon_1, \varepsilon_2>0$ such that
	\begin{equation}
		\tilde{\Gamma}_1:=\mathcal{E}_1^{-1}(1-\varepsilon_1), \qquad
		\tilde{\Gamma}_2:=\mathcal{E}_2^{-1}(1-\varepsilon_2),
	\end{equation}
	and the curves $\Gamma_1$, $\Gamma_2$, $\tilde{\Gamma}_1$, $\tilde{\Gamma}_2$ bound a region homeomorphic to a quadrilateral. An illustration is given in \cref{illustration3}.
	\begin{figure}
		\centering
		\includegraphics[width=0.8\textwidth]{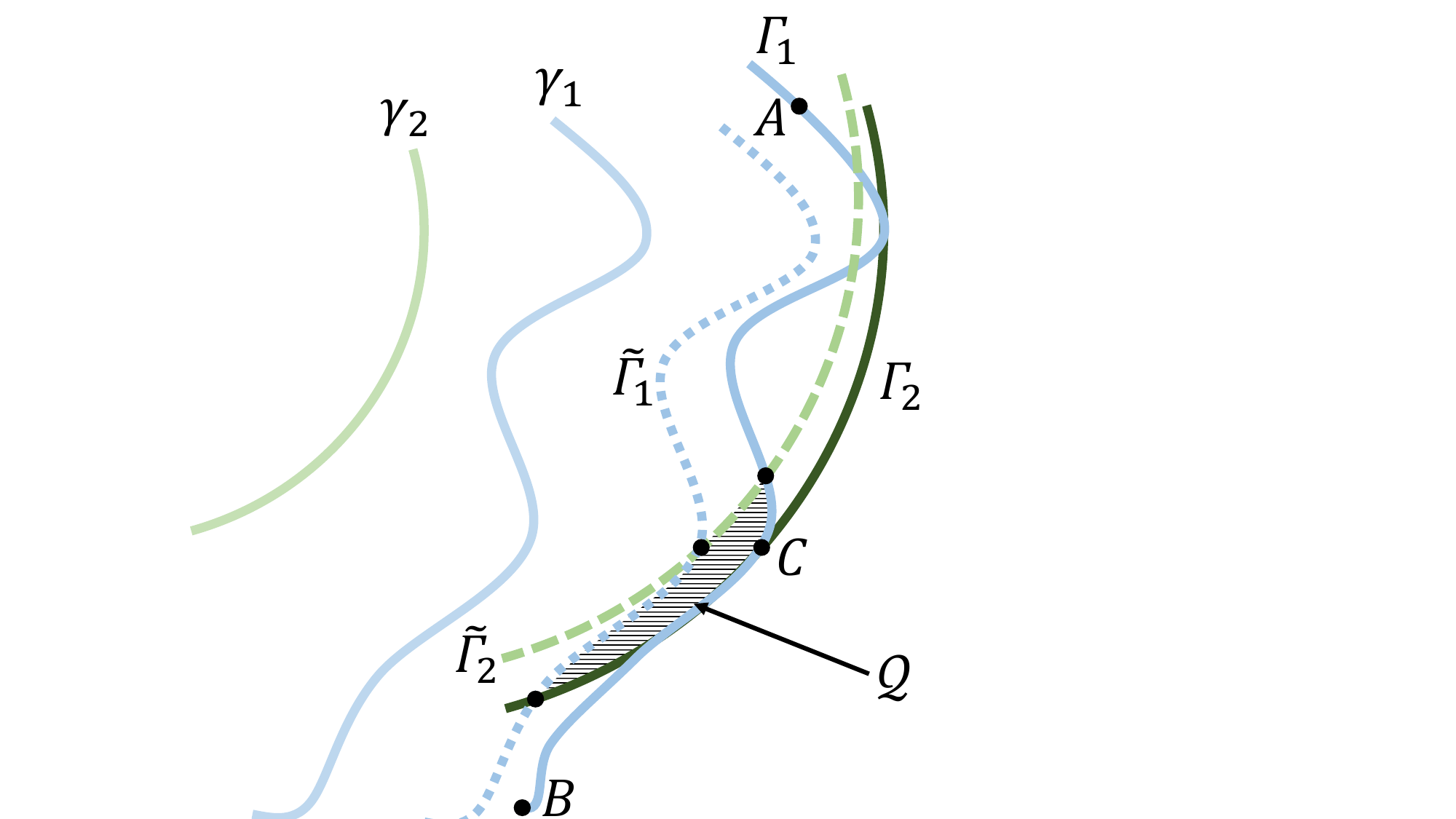}
		\caption{Case II, Subcase I: $\Gamma_1\cap \Gamma_2 \neq \emptyset$. Solid curves denote portions of the orbits $\gamma_i,\Gamma_i$ $(i=1,2)$, whereas dotted curves and dashed curves denote portions of the orbits $\tilde{\Gamma}_1$ and $\tilde{\Gamma}_2$, respectively.}
		\label{illustration3}
	\end{figure}
	
	\medskip
	\noindent\textbf{Subcase II.} $\Gamma_1\cap \gamma_2 \neq \emptyset$.
	
	\medskip
	Similarly, there exist $\varepsilon_1, \varepsilon_2>0$ such that
	\begin{equation}
		\tilde{\Gamma}_1:=\mathcal{E}_1^{-1}(1-\varepsilon_1), \qquad
		\tilde{\gamma}_2:=\mathcal{E}_2^{-1}(\varepsilon_2),
	\end{equation}
	and the curves $\Gamma_1$, $\gamma_2$, $\tilde{\Gamma}_1$, $\tilde{\gamma}_2$ bound a region homeomorphic to a quadrilateral. An illustration is provided in \cref{illustration4}.
	\begin{figure}
		\centering
		\includegraphics[width=0.8\textwidth]{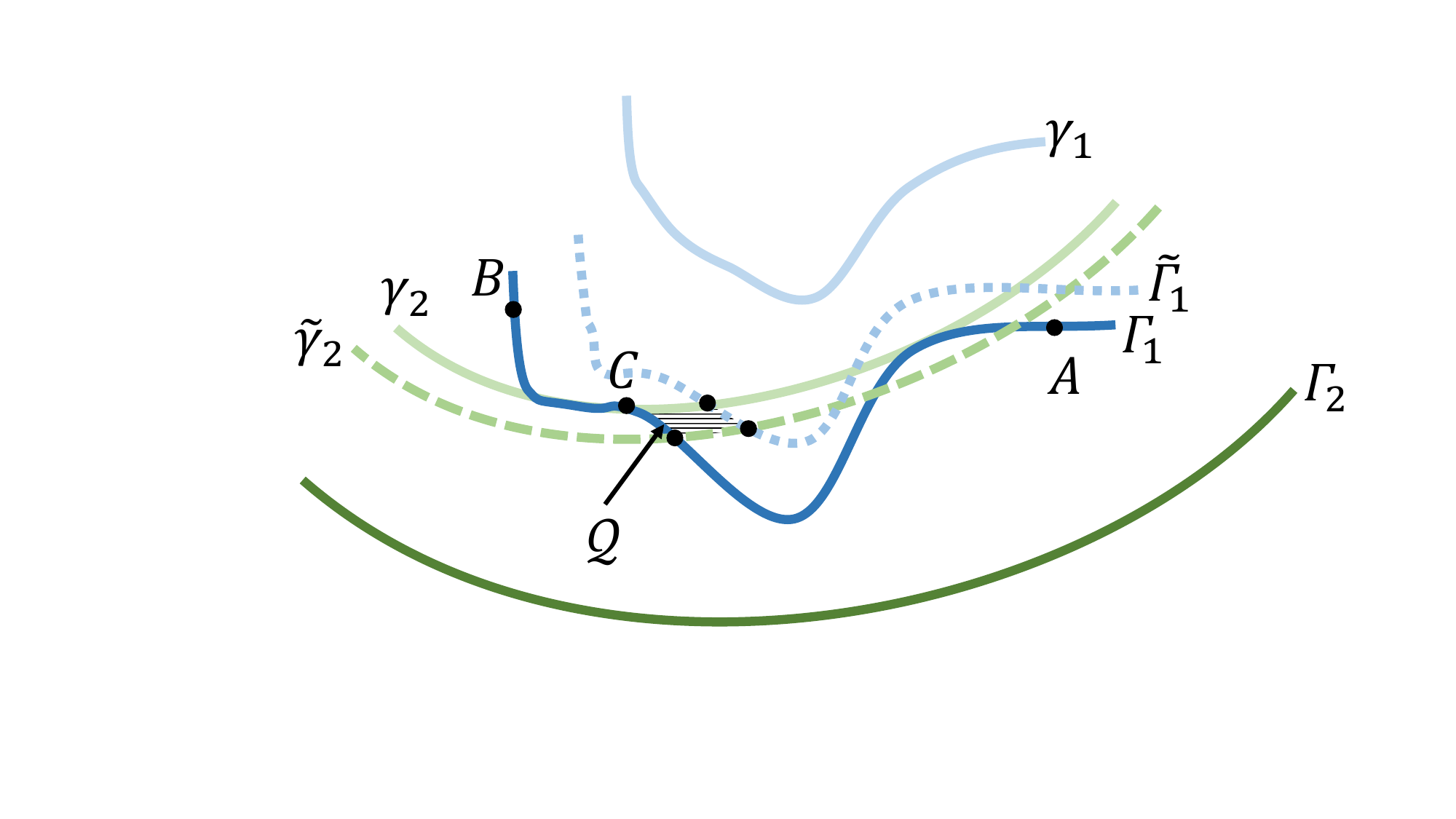}
		\caption{Case II, Subcase II: $\Gamma_1\cap \gamma_2 \neq \emptyset$. Solid curves denote portions of the orbits $\gamma_i,\Gamma_i$ $(i=1,2)$, whereas dotted curves and dashed curves denote portions of the orbits $\tilde{\Gamma}_1$ and $\tilde{\gamma}_2$, respectively.}
		\label{illustration4}
	\end{figure}
\end{proof}

\begin{proof}[Proof of \cref{MainTheorem1}]

	From \cref{lem1}, we assume that the four orbits defining the curvilinear quadrilateral $\mathcal{Q}$ are $\widetilde{\gamma_1}$, $\widetilde{\gamma_2}$, $\widetilde{\Gamma_1}$, and $\widetilde{\Gamma_2}$ and
	\begin{align}
		\mathcal{E}_1(\widetilde{\gamma_1})&=c_1,\\
		\mathcal{E}_2(\widetilde{\gamma_2})&=c_2,\\
		\mathcal{E}_1(\widetilde{\Gamma_1})&=C_1,\\
		\mathcal{E}_2(\widetilde{\Gamma_1})&=C_2.\\
	\end{align}
	
	We denote the edges and vertices of $\mathcal{Q}$ as illustrated in \cref{illustration5}:
	\begin{align}
		\mathcal{Q}^1 &:= \bar{\mathcal{Q}} \cap \widetilde{\Gamma_1},\\
		\mathcal{Q}^2 &:= \bar{\mathcal{Q}} \cap \widetilde{\Gamma_2},\\
		\mathcal{Q}^3 &:= \bar{\mathcal{Q}} \cap \widetilde{\gamma_1},\\
		\mathcal{Q}^4 &:= \bar{\mathcal{Q}} \cap \widetilde{\gamma_2},\\
		A &:= \mathcal{Q}^1 \cap \mathcal{Q}^4,\\
		B &:= \mathcal{Q}^1 \cap \mathcal{Q}^2,\\
		C &:= \mathcal{Q}^2 \cap \mathcal{Q}^3,\\
		D &:= \mathcal{Q}^3 \cap \mathcal{Q}^4,
	\end{align}
	where $\bar{\mathcal{Q}}$ denotes the closure of $\mathcal{Q}$.
	
	\begin{figure}
		\centering
		\includegraphics[width=0.8\textwidth]{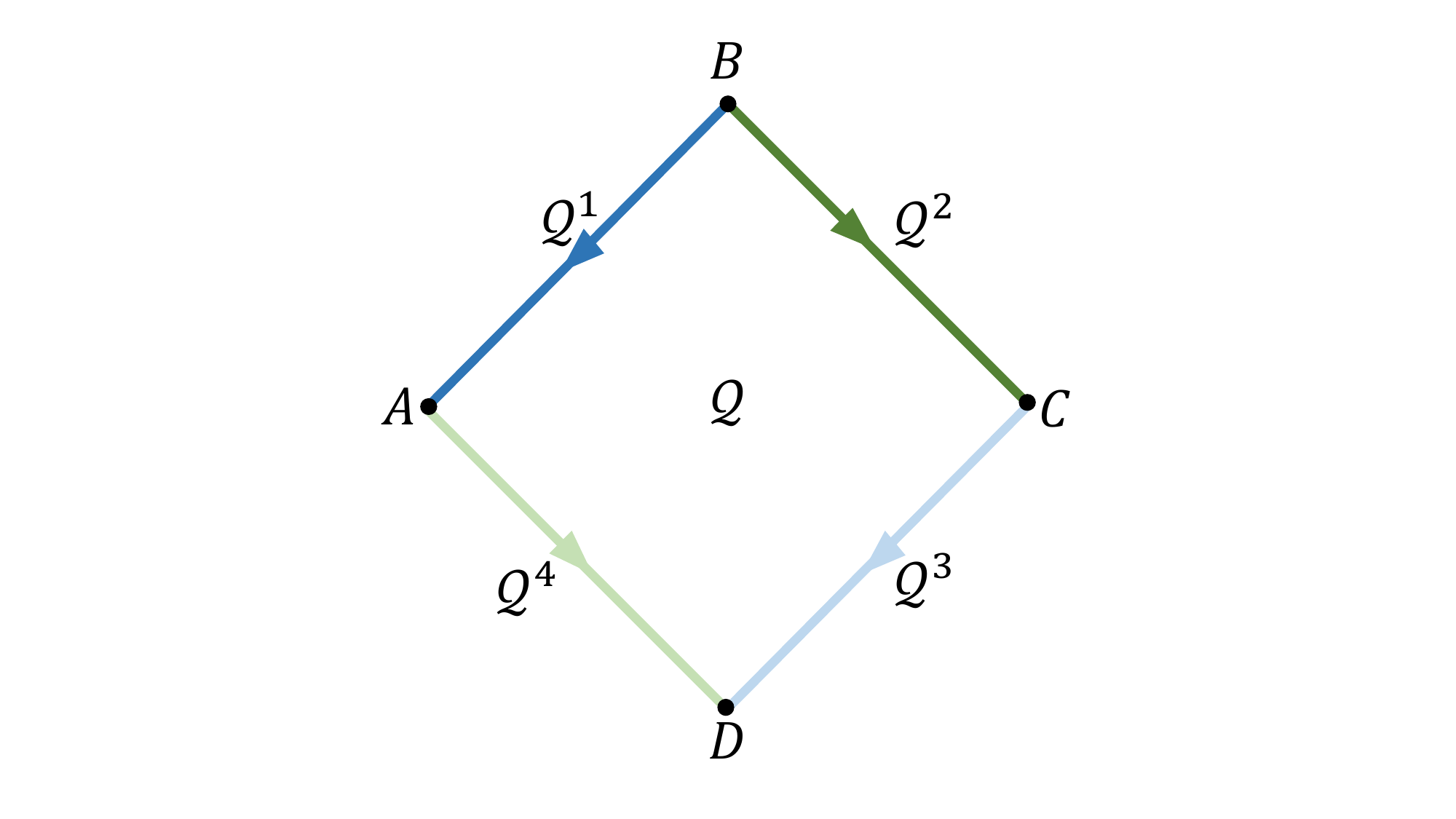}
		\caption{Edges and vertices of $\mathcal{Q}$. Arrows indicate the directions of the orbits.}
		\label{illustration5}
	\end{figure}

	In the following proof, we assume that, for each $i$ and $c\in (0,1)$, the period function $T_i(c)$ is strictly monotonically increasing with respect to $c$.

	Recall that within each period, the switching system \eqref{SwitchingSystem} evolves according to $\dot{x}=f_1(x)$ for a time interval of length $T_1$. 
	Consequently, the orbit starting from $D$ intersects $D$ $N_1$ times during the interval $[0,T_1]$, where
	\begin{equation}
		N_1 := \left\lfloor \frac{T_1}{T_1(c_1)} \right\rfloor.
	\end{equation}
	Similarly, the orbit starting from $A$ intersects $A$ $n_1$ times during $[0,T_1]$, where
	\begin{equation}
		n_1 := \left\lfloor \frac{T_1}{T_1(C_1)} \right\rfloor.
	\end{equation}
	
	Define
	\begin{equation}
		T_1^* := \frac{\, T_1(c_1)\, T_1(C_1)}{|T_1(c_1) - T_1(C_1)|}.
	\end{equation}
	Note that
	\begin{equation}
		\frac{5T_1^*}{T_1(c_1)}-\frac{5T_1^*}{T_1(C_1)}=5,
	\end{equation}
	implies that
	\begin{equation}
		\left\lfloor\frac{5T_1^*}{T_1(c_1)}\right\rfloor-\left\lfloor\frac{5T_1^*}{T_1(C_1)}\right\rfloor=5.
	\end{equation}
	It follows that whenever $T_1 \geq 5T_1^*$, we have
	\begin{equation}\label{Difference}
		N_1 - n_1 \geq 5.
	\end{equation}

	Consider the following two compact subsets of $\bar{\mathcal{Q}}$:
	\begin{align}
		\mathcal{Q}_1 &:= \bigl\{(x,y)\in\bar{\mathcal{Q}} : \frac{T_1}{\mathcal{T}_1(x,y)} \in [n_1+1, n_1+3] \bigr\},\\
		\mathcal{Q}_2 &:= \bigl\{(x,y)\in\bar{\mathcal{Q}} : \frac{T_1}{\mathcal{T}_1(x,y)} \in [N_1-2, N_1] \bigr\},
	\end{align}
	where $\mathcal{T}_i(x,y)$ denotes the period of the closed orbit of $\dot{x}=f_i(x)$ passing through $(x,y)$.  
	
	From \eqref{Difference}, it follows immediately that
	\begin{equation}
		\text{int}(\mathcal{Q}_1) \cap \text{int}(\mathcal{Q}_2) = \emptyset.
	\end{equation}
	
	We denote the edges of $\mathcal{Q}_1$ and $\mathcal{Q}_2$ as illustrated in \cref{illustration6}:
	\begin{align}
		\mathcal{Q}_1^1 &:= \big\{(x,y)\in\bar{\mathcal{Q}} : \frac{T_1}{\mathcal{T}_1(x,y)} = n_1+1\big\}, &
		\mathcal{Q}_1^2 &:= \mathcal{Q}_1 \cap \mathcal{Q}^2,\\
		\mathcal{Q}_1^3 &:= \big\{(x,y)\in\bar{\mathcal{Q}} : \frac{T_1}{\mathcal{T}_1(x,y)} = n_1+3\big\}, &
		\mathcal{Q}_1^4 &:= \mathcal{Q}_1 \cap \mathcal{Q}^4,\\
		\mathcal{Q}_2^1 &:= \big\{(x,y)\in\bar{\mathcal{Q}} : \frac{T_1}{\mathcal{T}_1(x,y)} = N_1-2\big\}, &
		\mathcal{Q}_2^2 &:= \mathcal{Q}_2 \cap \mathcal{Q}^2,\\
		\mathcal{Q}_2^3 &:= \big\{(x,y)\in\bar{\mathcal{Q}} : \frac{T_1}{\mathcal{T}_1(x,y)} = N_1\big\}, &
		\mathcal{Q}_2^4 &:= \mathcal{Q}_2 \cap \mathcal{Q}^4.
	\end{align}
	\begin{figure}
		\centering
		\includegraphics[width=0.8\textwidth]{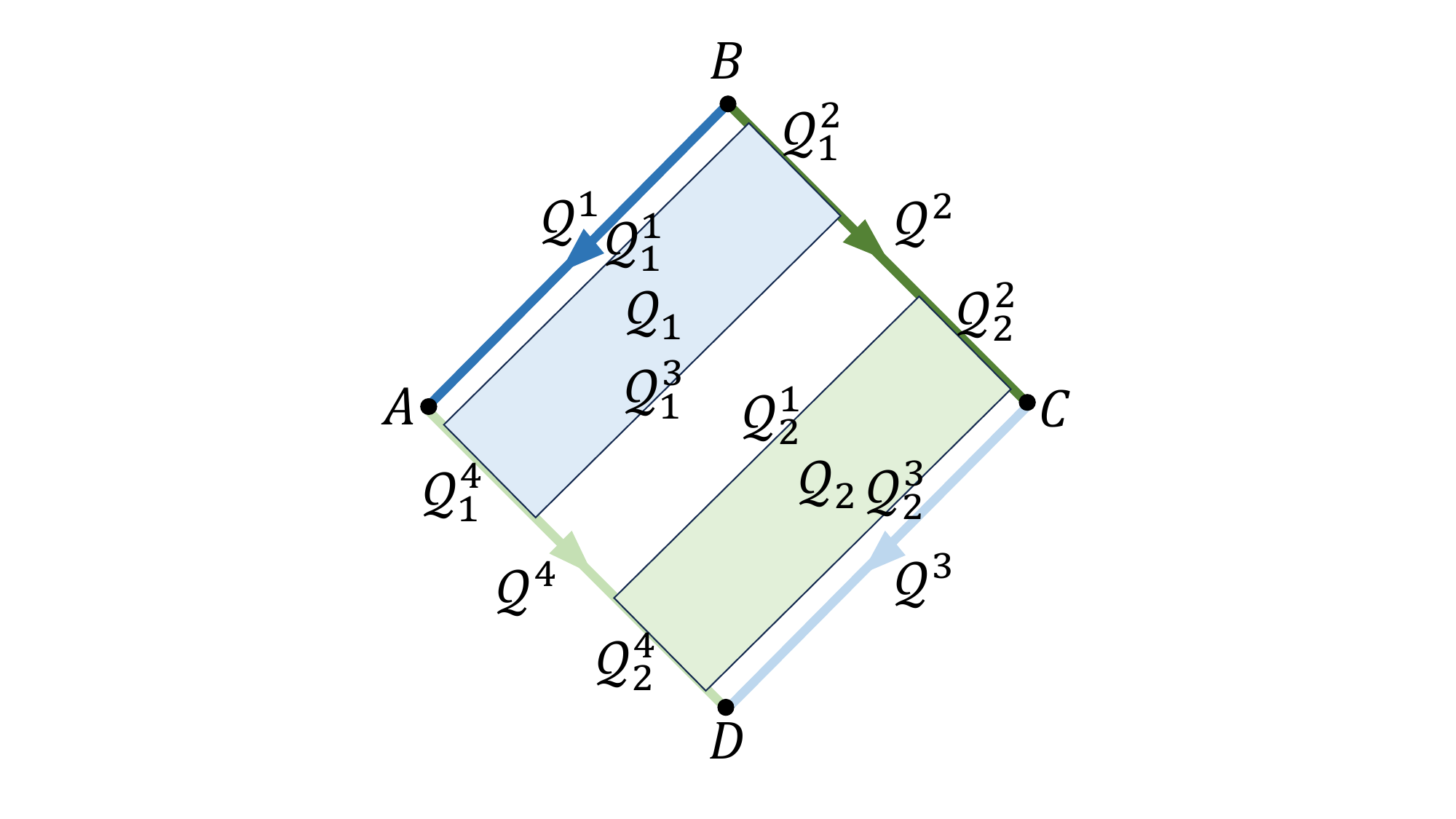}
		\caption{Edges of $\mathcal{Q}_1$ and $\mathcal{Q}_2$.}
		\label{illustration6}
	\end{figure}

	Next, define the ``semi''-Poincar\'e map
	\begin{equation}
		P_1 : (x_0,y_0) \mapsto \big(x_1(T_1),y_1(T_1)\big),
	\end{equation}
	where $\big(x_1(T_1),y_1(T_1)\big)$ denotes the solution of $\dot{x}=f_1(x)$ at time $T_1$ with initial condition $(x_0,y_0)$.
	
	The image of $\mathcal{Q}_1$ under $P_1$ exhibits a ``crossing'' behavior relative to $\mathcal{Q}_1$, although not in the classical sense of a topological horseshoe. More precisely, for any  path $l_1$ connecting $\mathcal{Q}_1^1$ and $\mathcal{Q}_1^3$, there exists a closed subinterval $[\alpha_1,\beta_1] \subset [0,1]$ such that $P_1\bigl(l_1([\alpha_1,\beta_1])\bigr)$ is contained in $\mathcal{Q}_1$ and
	connects $\mathcal{Q}_1^2$ and $\mathcal{Q}_1^4$.
	
	Since the image of $\mathcal{Q}_1$ under $P_1$ may ``corss'' itself multiple times, we define $\widetilde{\mathcal{Q}}_1$ to a connected component of the following set $U_1$:
	\begin{equation}
		U_1:= \bigcup \big\{P_1\big(l[\alpha,\beta]\big): l\subset\mathcal{Q}_1~ \text{connects}~ Q_1^1 ~\text{and}~ Q_1^3,~ P_1\big(l(\alpha)\big)\in Q_1^2,~ P_1\big(l(\beta)\big)\in Q_1^4~ s.t.~ P_1\big(l[\alpha,\beta]\big) \subset \mathcal{Q}_1\big\}.
	\end{equation}
	We denote the edges of $\widetilde{\mathcal{Q}}_1$  as illustrated in \cref{illustration7}:
	\begin{align}
		\widetilde{\mathcal{Q}}_1^1 &:= \widetilde{\mathcal{Q}}_1\cap P_1(\mathcal{Q}_1^4),\\
		\widetilde{\mathcal{Q}}_1^2 &:= \widetilde{\mathcal{Q}}_1\cap \mathcal{Q}_1^2,\\
		\widetilde{\mathcal{Q}}_1^3 &:=  \widetilde{\mathcal{Q}}_1\cap P_1(\mathcal{Q}_1^2),\\
		\widetilde{\mathcal{Q}}_1^4 &:= \widetilde{\mathcal{Q}}_1\cap \mathcal{Q}_1^4.
	\end{align}
	\begin{figure}
		\centering
		\includegraphics[width=0.8\textwidth]{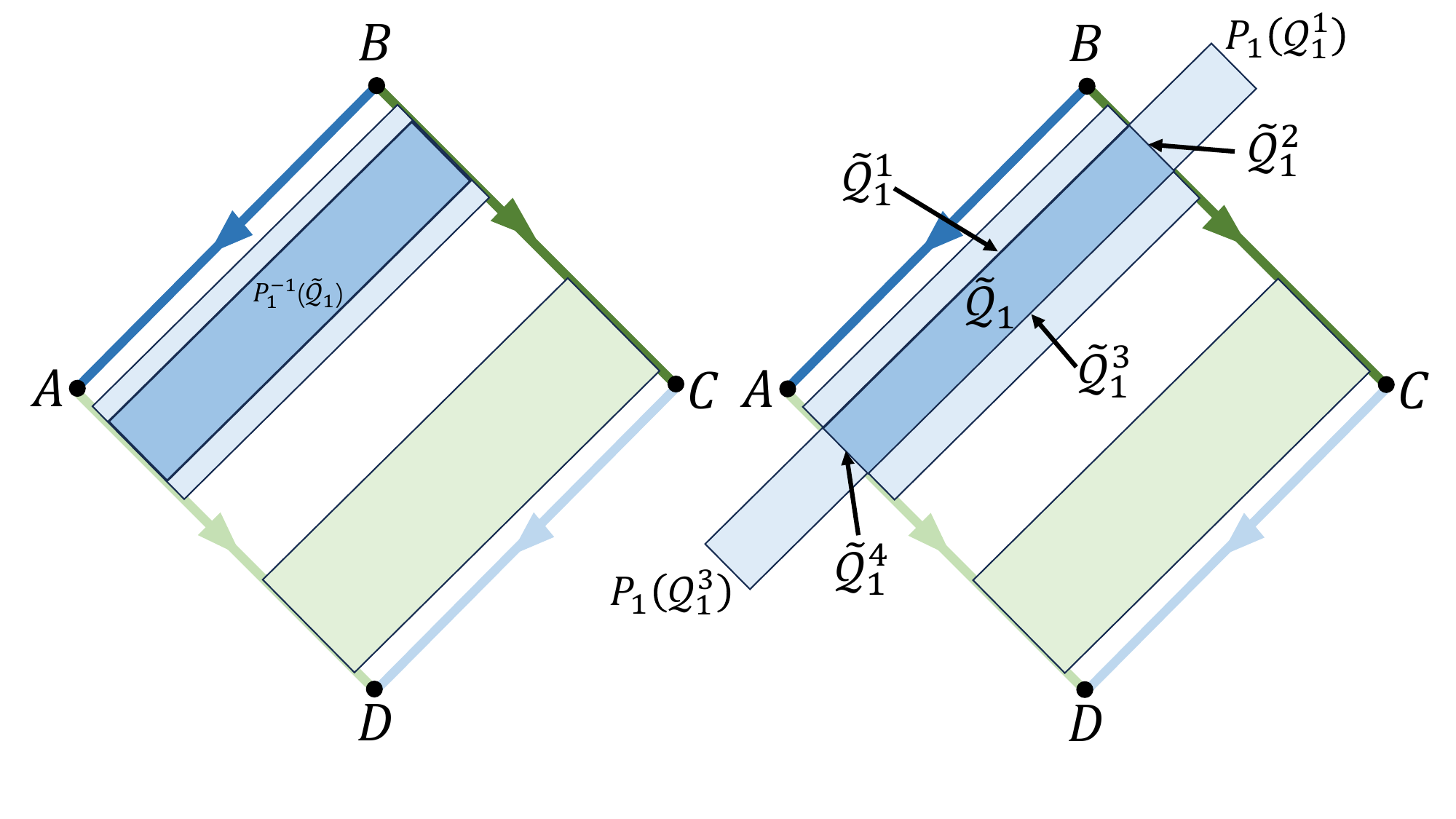}
		\caption{Preimage of $\tilde{\mathcal{Q}}_1$ and edges of $\tilde{\mathcal{Q}}_1$. The figure illustrates the case in which $P_1(\mathcal{Q}_1)$ ``crosses'' $\mathcal{Q}_1$ exactly once. (Left) Relative position of $P_1^{-1}(\tilde{\mathcal{Q}}_1)$ and $\tilde{\mathcal{Q}}_1$. (Right) $P_1(\mathcal{Q}_1)$ ``crosses'' $\mathcal{Q}_1$ once; in this case, $\tilde{\mathcal{Q}}_1 = P_1(\mathcal{Q}_1)\cap \mathcal{Q}_1$.}
		\label{illustration7}
	\end{figure}

	Next, set
	\begin{equation}
		T_2^* := \frac{\, T_2(c_2)\, T_2(C_2)}{|T_2(c_2)-T_2(C_2)|}.
	\end{equation}

	Define the ``semi''-Poincar\'e map associated with the second subsystem:
	\begin{equation}
		P_2 : (x_0,y_0) \mapsto \big(x_2(T_2), y_2(T_2)\big),
	\end{equation}
	where $\big(x_2(T_2),y_2(T_2)\big)$ is the solution of $\dot{x}=f_2(x)$ at time $T_2$ with initial condition $(x_0,y_0)$.
	
	It follows that, whenever $T_2 > 3T_2^*$, for any path $l_2$ connecting $\widetilde{\mathcal{Q}}_1^2$ and $\widetilde{\mathcal{Q}}_1^4$, there exists a closed subinterval $[\alpha_2,\beta_2]\subset[0,1]$ such that $P_2\bigl(l_2([\alpha_2,\beta_2])\bigr)$ is contained in $\mathcal{Q}$ and
	connects $\mathcal{Q}^1$ and $\mathcal{Q}^3$.  
	
	In other words, we have
	\begin{equation}
		P_2 (\widetilde{\mathcal{Q}}_1) \mapsto P_1^{-1}( \widetilde{\mathcal{Q}}_1).
	\end{equation}
	
	Define $\widetilde{\mathcal{Q}}_2$ to a connected component of the following set $U_2$:
	\begin{equation}
		U_2:= \bigcup \big\{P_2\big(l[\alpha,\beta]\big): l\subset\mathcal{Q}_2~ \text{connects}~ Q_2^1 ~\text{and}~ Q_2^3,~ P_2\big(l(\alpha)\big)\in Q_2^2,~ P_2\big(l(\beta)\big)\in Q_2^4~ s.t.~ P_2\big(l[\alpha,\beta]\big) \subset \mathcal{Q}_2\big\}.
	\end{equation}

	Similarly, for $i,j=1,2$, the following relation holds:
	\begin{equation}
		P_2 \circ P_1\big(P_1^{-1}(\widetilde{\mathcal{Q}_i})\big) \mapsto P_1^{-1}(\widetilde{\mathcal{Q}_j}).
	\end{equation}
	
	For reader's convenience, we provide illustrations in \cref{illustration11,illustration12} to demonstrate the formation and structure of the topological horseshoe.
	In \cref{illustration11}, for $i=1,2$
	\begin{align}
		\mathcal{Q}_i&=A_iB_iC_iD_i\\
		\widetilde{\mathcal{Q}_i}&=E_iF_iG_iH_i
	\end{align}

	\begin{figure}
		\centering
		\includegraphics[width=\textwidth]{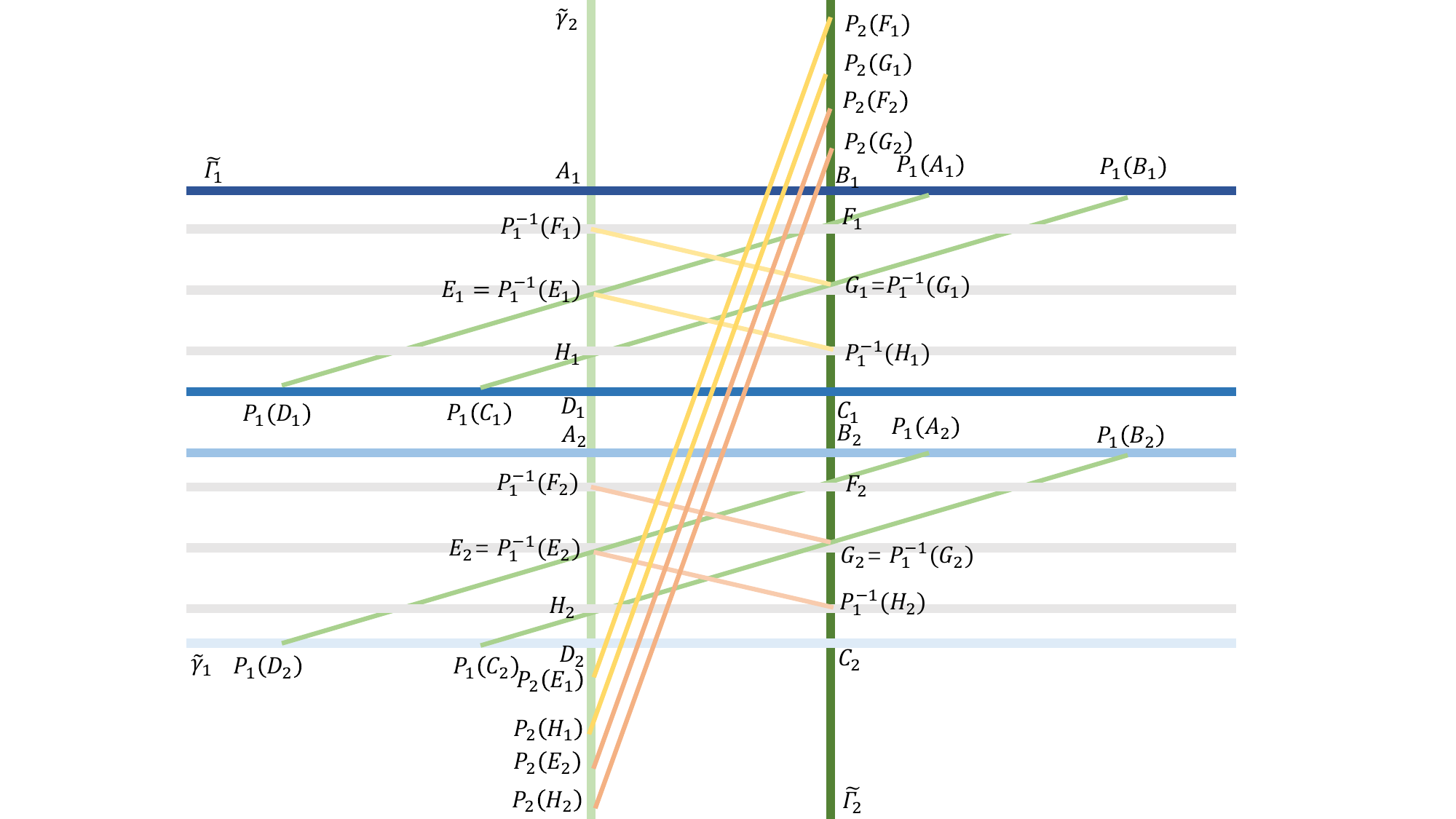}
		\caption{Illustration of the formation of the topological horseshoe. The figure illustrates the case in which $P_1(\mathcal{Q}_i)$ ``crosses'' $\mathcal{Q}_i~(i=1,2)$ exactly once. In this case, $P_1$ is the identity map when restricted to the orbits containing $E_i$ and $G_i$ $(i=1,2)$. Moreover, for each $i$, the points $E_i$ and $G_i$ lie on the same orbit.}
		\label{illustration11}
	\end{figure}

	\begin{figure}
		\centering
		\includegraphics[width=\textwidth]{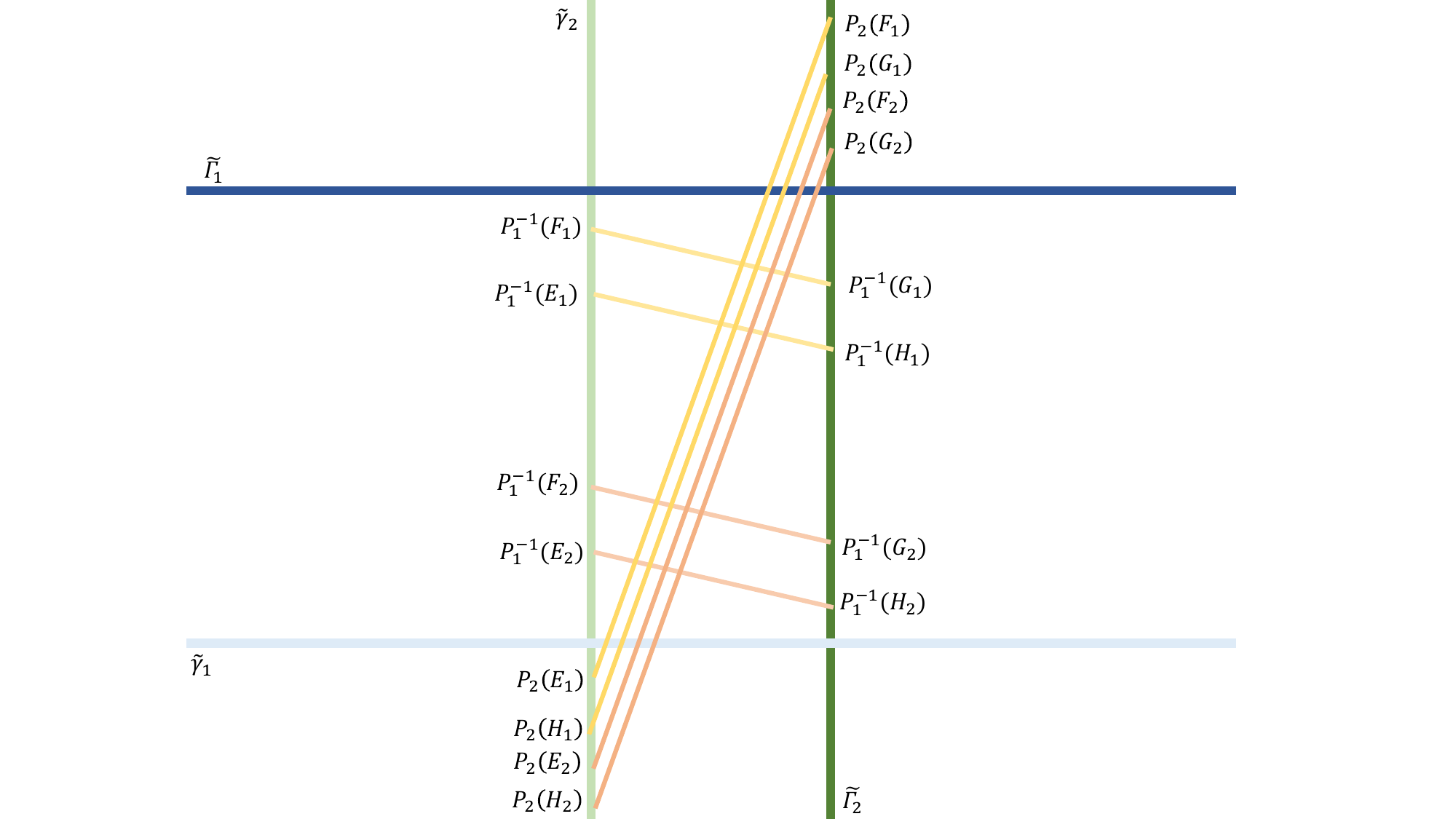}
		\caption{Illustration of the topological horseshoe:\\$P_2 \circ P_1\big(P_1^{-1}(E_iF_iG_iH_i)\big)\mapsto P_1^{-1}(E_jF_jG_jH_j),\quad i,j=1,2$.}
		\label{illustration12}
	\end{figure}
	
	It is noteworthy that this is a periodic system, so $P_1\circ P_2$ is also a Poincar\'e map. It follows that if 	\begin{equation}
		T_1 > 3T_1^* \quad \text{and} \quad T_2 > 5T_2^*,
	\end{equation}
	$P_1\circ P_2$ will exhibit a topological horseshoe.

	Therefore, the Poincar\'e map associated with the switching system has a topological horseshoe if there exist indices $i\neq j$ such that
	\begin{equation}
		T_i > 5T_i^* \quad \text{and} \quad T_j > 3T_j^*,
	\end{equation}
	where
	\begin{equation}
		T_k=\frac{T_k(c_k)T_k(C_k)}{|T_k(C_k)-T_k(c_k)|},\quad k=1,2.
	\end{equation}

\end{proof}		

\begin{remark}
	We do not require $\mathcal{Q}_1^3$ to be different from $\mathcal{Q}_2^1$. As can be seen from \cref{illustration11,illustration12}, even if $\mathcal{Q}_1^3=\mathcal{Q}_2^1$, we have
	\begin{equation}
		P_1^{-1}(E_1F_1G_1H_1)\cap P_1^{-1}(E_2F_2G_2H_2)=\emptyset.
	\end{equation}
\end{remark}

%\begin{remark}
%	Our results are inspired by the work in Ref.~\cite{Barrientos2017}. Instead of studying some specific model like in Ref.~\cite{Barrientos2017}, we discuss general conservative systems, and provide several  geometric conditions for the emergence of chaos. Besides, by conducting a more detailed investigation into the formation mechanism of the topological horseshoe, we provide a better estimation of the lower bounds for the duration of each system and point out that the lower bounds are not same, which is somewhat counterintuitive. 
%\end{remark}

\begin{remark}
	It can be verified that other types of monotonicity of the period functions $T_i(c)$ do not affect the validity of our result. The assumption that $T_i$ is strictly increasing is imposed solely for  convenience of the proof.
\end{remark}

\section{Chaos induced by periodic switching between non-isochronous planar Hamiltonian systems}\label{Hamiltonian_systems}

Before introducing our main results in this section, let us recall the definition of the period function in Hamiltonian system, which plays a crucial role in the conditions of our theorem.

\begin{definition}\cite{Chow1986,Chicone1987}
	Consider a planar Hamiltonian system
	\begin{equation}\label{PF}
		\dot{x} = f(x), \quad x \in \mathbb{R}^2,
	\end{equation}
	and suppose the system has a family of periodic orbits parameterized by a scalar quantity (typically an energy level or a first integral).

	Assume there exists a first integral \( H(x) \) such that each regular level set
	\begin{equation}
		\Gamma_h := \{ x \in \mathbb{R}^2 \mid H(x) = h \}
	\end{equation}
	is a closed orbit for \(h\) in some interval \(I\).
	
	The period function $T:I\to\mathbb{R}^+$ assigns to each periodic orbit its minimal period. More precisely,
	\begin{equation}
		T(h)=\inf\{t>0:\phi(t,x_0)=x_0,H(x_0)=h\},
	\end{equation}
	where $\phi(t,x)$ denotes the flow associated with \eqref{PF}.
\end{definition}

\subsection{General planar Hamiltonian system}

Consider the following  planar autonomous Hamiltonian system 
\begin{equation}\label{Hamiltonian}
	\begin{cases}
		\dot{x}=H_y(x,y),\\
		\dot{y}=-H_x(x,y).
	\end{cases}
\end{equation}

In a neighborhood of the center, the period function can be characterized as the derivative of the area enclosed by a periodic orbit with respect to the energy level; see Ref.~\cite{Manosas2002}.

\begin{lemma}\cite{Manosas2002}\label{ManosasThmA} 
	Let $H$ be an analytic function with $H(0,0)=0$ and assume that the Hamiltonian system \eqref{Hamiltonian} has a nondegenerate center at the origin. Then, in a neighborhood of $h=0$, the functions $A$ and $T$ are analytic and satisfy the relation 
	\begin{equation}
		A'(h)=T(h),
	\end{equation}
	where $T(h)$ is the period function, and $A(h)$ gives the area of the region surrounded by $\gamma_h$.
\end{lemma} 

According to \cref{ManosasThmA}, in a neighborhood of the center, we have
\begin{equation}
	T'(h)=A''(h).
\end{equation}

Consider the $T$-periodic switching system in $\mathbb{R}^2$, 
\begin{equation}\label{SwitchingSystem2}
	\begin{cases}
		\dot{x}&=\begin{cases}
			\frac{\partial}{\partial y} H_1(x,y),\qquad \quad nT&\leq t<nT+T_1,\\
			\frac{\partial}{\partial y} H_2(x,y),\qquad nT+T_1&\leq t<nT+T_1+T_2,
		\end{cases}\\
		\dot{y}&=\begin{cases}
			-\frac{\partial}{\partial x} H_1(x,y),\qquad \quad nT&\leq t<nT+T_1,\\
			-\frac{\partial}{\partial x} H_2(x,y),\qquad nT+T_1&\leq t<nT+T_1+T_2,\\
		\end{cases}
	\end{cases}
\end{equation}
where $T_1+T_2=T$ and $H_i(x,y)$ is analytic for $i=1,2$.
$T_i$ is the dwell time associated with the Hamiltonian subsystems:
\begin{equation}\label{SwitchingSystem2_1}
	\begin{cases}
		\dot{x}&=\frac{\partial}{\partial y} H_i(x,y),\\
		\dot{y}&=-\frac{\partial}{\partial x} H_i(x,y).
	\end{cases}
\end{equation}

Similar to \cref{hypothesis1}, we propose the following hypothesis.
\begin{Hypothesis}\label{hypothesis3}
	For each $i=1,2$, denote by $T_i(c)$ the period of the closed orbit of \eqref{SwitchingSystem2} corresponding to the Hamiltonian value $c$. Suppose that, for each $i$, there exist two closed orbits $\gamma_i$ and $\Gamma_i$ associated with the Hamiltonian values $c_i$ and $C_i$, respectively, with $c_i<C_i$. Assume that $\Gamma_i$ contains exactly one fixed point of \eqref{SwitchingSystem2_1} in its interior.
\end{Hypothesis}

\begin{notation}
	Let
	\begin{equation}
		\mathcal{R}_i=\big\{(x,y)\in\mathbb{R}^2 : c_i<H_i(x,y)<C_i\big\}
	\end{equation}
	denote the open annular region bounded by $\gamma_i$ and $\Gamma_i$.
\end{notation}

\begin{theorem}\label{MainTheorem2}	
	Assume that \cref{hypothesis3} holds. Suppose further that the following  conditions are satisfied:
	\begin{enumerate}[label=(\roman*)]
		\item There exist indices $i\neq j$ such that either
		$\gamma_i\cap\mathcal{R}_j\neq\emptyset$ and
		$\gamma_i\not\subset\mathcal{R}_j$, or
		$\Gamma_i\cap\mathcal{R}_j\neq\emptyset$ and
		$\Gamma_i\not\subset\mathcal{R}_j$. 
		\item \begin{equation}
			A_i''(c)\neq 0, \quad c_i<c<C_i,\quad i=1,2
		\end{equation} 
		where $A_i(c)$ gives the area of the region surrounded by the closed orbit of \eqref{SwitchingSystem2_1} which is corresponding to the energy level $c$.
	\end{enumerate}
	Then, there exist constants $T_1^*>0$ and $T_2^*>0$ such that, whenever
	\begin{equation}\label{LowerBound}
		T_i \ge 5T_i^*, \quad T_j \ge 3T_j^*, \quad i,j\in{1,2},\ i\neq j,
	\end{equation}
	the time-$T$ Poincar\'e map associated with the switching system \eqref{SwitchingSystem2} exhibits a topological horseshoe, where $T=T_1+T_2$.
\end{theorem}

\begin{remark}
	The constants $T_k^*~(k=1,2)$ in \eqref{LowerBound} are defined by
	\begin{equation}
		T_k=\frac{T_k(\tilde{c}_k)T_k(\tilde{C}_k)}{|T_k(\tilde{C}_k)-T_k(\tilde{c}_k)|},\quad k=1,2.
	\end{equation}
	Here, \( \tilde{c}_k \) and \( \tilde{C}_k \) \((k=1,2)\) denote the energy levels (Hamiltonian values) corresponding to four closed orbits that form a curvilinear quadrilateral , and satisfy
	\begin{equation}
		c_k \le \tilde{c}_k < \tilde{C}_k \le C_k\quad (k=1,2).
	\end{equation}
\end{remark}

The condition of the theorem above may be hard to checked analytically, but can be verified through numerical methods. Fortunately, for one-degree-of-freedom mechanical systems, also referred to as planar Newtonian Hamiltonian systems, the period--energy relationship has been extensively studied, and a well-developed theoretical framework is available to address this issue.

\subsection{ One-degree-of-freedom mechanical system (Newtonian Hamiltonian system)}

Several sufficient criteria ensuring the monotonicity of the period function have been established.
Among them, the result of Chow and Wang plays a fundamental role in the study of monotonicity of the energy-period function \cite{Chow1986}.

Consider the scalar second-order differential equation
\begin{equation}\label{Chow1_1}
	\ddot{x}+g(x)=0
\end{equation} 
where $g$ is smooth for all $x\in\mathbb{R}$.
Equation \eqref{Chow1_1} is equivalent to the Newtonian Hamiltonian system
\begin{equation}
	\begin{cases}
		\dot{x}&=y,\\
		\dot{y}&=-g(x).
	\end{cases}
\end{equation}
with Hamiltonian
\begin{equation}
	H(x,y)=\frac{1}{2}y^2+G(x),
\end{equation}
where
\begin{equation}
	G(x)=\int_0^x g(\xi)\text{d}\xi.
\end{equation}

Assume that there exist $a<0<b$ such that 
\begin{equation}
	G(a)=G(b)=c,~G(x)<c~\text{for~all}~a<x<b,
\end{equation}
and $g(a)g(b)\neq 0$.

Then the energy level set $\{H=c\}$ contains a periodic orbit in the phase plane, intersecting the $x$-axis at $(a,0)$ and $(b,0)$.
Denote by $T(c)$ the period function.

When restricting attention to a single Newtonian Hamiltonian system, it is natural to assume that
\begin{equation}
	g(0)=0,
\end{equation}
and the periodic orbits contain only one fixed point in their interiors.

We recall the following hypothesis introduced in Ref. \cite{Chow1986}.
\begin{Hypothesis}\cite{Chow1986}\label{H1}
	There exist $-\infty\leq a^* < 0 < b^*\leq +\infty$, an integer $N\geq 0$ and a positive smooth function $h(x)$ such that 
	\begin{equation}\label{2_1}
		g(x)=x^{2N+1}h(x),\qquad a^*<x<b^*
	\end{equation}
	and
	\begin{equation}
		0<G(a^*)=G(b^*)=c^*\leq +\infty.
	\end{equation}
\end{Hypothesis}
As observed in Ref.~\cite{Chouikha1999}, condition \eqref{2_1} is equivalent to
\begin{equation}
	xg(x)>0,\qquad x\in(a^*,b^*)\backslash\{0\}.
\end{equation}

Under \cref{H1}, the following result provides a criterion  for the monotonicity of the period function.
\begin{lemma}\cite{Chow1986}\label{Chow}
	Suppose \cref{H1} holds. If $g'(0)>0$ and 
	\begin{equation}\label{Chow_eq1}
		H(x)=g(x)^2+\frac{g''(x)}{3g'(0)^2}g(x)^3-2G(x)g'(x),
	\end{equation}
	satisfies
	\begin{equation}
		H(x)>0~(\text{respectively} <0), \quad x\neq0,\quad a^*<x<b^*,
	\end{equation}
	then the period function $T(c)$ is strictly increasing (respectively strictly decreasing) on $(0,c^*)$, that is,
	\begin{equation}
		T'(c)>0 \quad(\text{respectively} <0),\quad 0<c<c^*.
	\end{equation}
\end{lemma}

%%%%%%%%%%%%%%%%%%%%%%%%%%%%%%%%%%%%%%%%%%%%%%%%%%%%%%%%%%%%%%%%%%%%%%%

When dealing with the switching system, it is unnecessary to require that the two Newtonian Hamiltonian subsystems share a common center. Accordingly, \cref{H1} and \cref{Chow} must be reformulated so as to accommodate the situation in which the critical point associated with the period function is located at an arbitrary position. Also, as can be seen from \eqref{Chow_eq1}, $g\in C^2(\mathbb{R})$ is enough for our result.

%%%%%%%%%%%%%%%%%%%%%%%%%%%%%%%%%%%%%%%%%%%%%%%%%%%%%%%%%%%%%%%%%%%%%%%

Consider the pair of scalar equations
\begin{equation}
	\ddot{x}+g_i(x)=0,\quad i=1,2,
\end{equation}
where $g_i\in C^2(\mathbb{R})$.
Suppose that, for each $i$, there exists exactly one critical point in the interior of the periodic orbits of the system, denoted by $({\bf x}_i,0)$; 
that is,
\begin{equation}
	g_i({\bf x}_i)=0,\quad i=1,2.
\end{equation}
Define
\begin{equation}
	G_i(x)=\int_{{\bf x}_i}^x g(\xi) \text{d} \xi.
\end{equation}

Similar to \cref{H1}, we propose the following the hypothesis.
\begin{Hypothesis}\label{H2}
	For each $i=1,2$, there exist $-\infty\leq \alpha_i < {\bf x}_i < \beta_i\leq +\infty$, such that 
	\begin{equation}
		(x-{\bf x}_i)g_i(x)>0,\quad x\in(\alpha_i,\beta_i)\backslash \{{\bf x}_i\},
	\end{equation}
	and
	\begin{equation}
		0<G(\alpha_i)=G(\beta_i)\leq +\infty.
	\end{equation}
\end{Hypothesis}

As a direct consequence of \cref{Chow}, applying the translation 
\begin{equation}
	y=x-{\bf x}_i,
\end{equation}
%which shifts the equilibrium to the origin and reduces the system to a standard form, 
we obtain the following corollary.
\begin{cor}\label{cor3}
	Suppose \cref{H2} holds. If, for each $i=1,2$, $g_i'({\bf x}_i)>0 $ and 
	\begin{equation}\label{ChowCondition}
		H_i(x)=g_i(x)^2+\frac{g_i''(x)}{3g_i'({\bf x}_i)^2}g_i(x)^3-2G_i(x)g_i'(x)>0\quad(\text{or} <0), 
	\end{equation}
	satisfies
	\begin{equation}
		H_i(x)>0~(\text{respectively} <0),x\neq {\bf x}_i, \alpha_i<x<\beta_i,
	\end{equation}
	then the corresponding period function $T_i(c)$ is strictly increasing (respectively strictly decreasing) on $(0,c_i^*)$, that is,
	\begin{equation}
		T_i'(c)>0 \quad(\text{or}  <0),\quad 0<c<c^*.
	\end{equation}
\end{cor}

%\subsubsection{Results on planar systems switching between two Newtonian Hamiltonian subsystems}

%Combining \cref{MainTheorem1} with \cref{cor3}, we obtain the following result.

Consider the following $T$-periodic switching system in $\mathbb{R}^2$, 
\begin{equation}\label{SwitchingSystem3}
	\begin{cases}
		\dot{x}&=y,\\
		\dot{y}&=\begin{cases}
			-g_1(x), \qquad  nT & \leq t<nT+T_1,\\
			-g_2(x), \quad  nT+T_1 &\leq t<nT+T_1+T_2.
		\end{cases}
	\end{cases}
\end{equation}
where $T_1+T_2=T$.
$T_i$ is the dwell time associated with the Newtonian Hamiltonian subsystems:
\begin{equation}\label{SwitchingSystem3_1}
	\begin{cases}
		\dot{x}&=y,\\
		\dot{y}&=-g_i(x).
	\end{cases}
\end{equation}
Similar to \cref{hypothesis1}, we propose the following hypothesis.
\begin{Hypothesis}\label{hypothesis4}
	For each $i$, $g_i\in C^2(\mathbb{R})$ and there exist two closed orbits $\gamma_i$ and $\Gamma_i$ associated with energy levels $c_i$ and $C_i$, respectively, with $c_i<C_i$. Assume that $\Gamma_i$ contains exactly one critical point $({\bf x}_i,0)$ of \eqref{SwitchingSystem3} in its interior.
\end{Hypothesis}

%	
%	For each $i=1,2$, denote by $T_i(c)$ the period of the closed orbit of \eqref{SwitchingSystem3} corresponding to the energy level $c$. 

%
%	Define
%	\begin{equation}
	%		G_i(x)=\int_{{\bf x}_i}^x g(\xi) \text{d} \xi,
	%	\end{equation}
%	and 
\begin{notation}
	Let
	\begin{equation}
		\mathcal{R}_i=\big\{(x,y)\in\mathbb{R}^2 : c_i<\mathcal{E}_i(x,y)<C_i\big\}
	\end{equation}
	be the open annular region bounded by $\gamma_i$ and $\Gamma_i$, where $\mathcal{E}_i(x,y)=\frac12 y^2+G_i(x)$.
\end{notation}

Combining \cref{MainTheorem1} with \cref{cor3}, we obtain the following result.
\begin{theorem}\label{MainTheorem3}	
	Assume that \cref{H2} and \cref{hypothesis4} hold. Suppose further that the following  conditions are satisfied:
	\begin{enumerate}[label=(\roman*)]
		\item There exist indices $i\neq j$ such that either
		$\gamma_i\cap\mathcal{R}_j\neq\emptyset$ and
		$\gamma_i\not\subset\mathcal{R}_j$, or
		$\Gamma_i\cap\mathcal{R}_j\neq\emptyset$ and
		$\Gamma_i\not\subset\mathcal{R}_j$. 
		\item For $i=1,2$, there exist $-\infty\leq \alpha_i < 0 < \beta_i\leq +\infty$, such that 
		\begin{equation}
			(x-{\bf x}_i)g_i(x)>0,\quad x\in(\alpha_i,\beta_i)\backslash \{{\bf x}_i\},
		\end{equation}
		and
		\begin{equation}
			0<G(\alpha_i)=G(\beta_i)\leq +\infty.
		\end{equation}
		\item 	For $i=1,2$, $g_i'({\bf x}_i)>0 $.
		\item For all $x_i\in(\alpha_i,\beta_i)\backslash \{{\bf x}_i\}$,
		\begin{equation}
			H_i(x)\neq 0,
		\end{equation}where $H_i(\cdot)$ is defined by \eqref{ChowCondition}.
	\end{enumerate}
	Then, there exist constants $T_1^*>0$ and $T_2^*>0$ such that, whenever
	\begin{equation}
		T_i \ge 5T_i^*, \quad T_j \ge 3T_j^*, \quad i,j\in{1,2},\ i\neq j,
	\end{equation}
	the time-$T$ Poincar\'e map associated with the switching system \eqref{SwitchingSystem3} exhibits a topological horseshoe, where $T=T_1+T_2$.
\end{theorem}

\begin{remark}
	Additional sufficient conditions ensuring that the period function is nondecreasing can be found in Ref.~\cite{Chouikha1999}.  In the present context, such criteria must be examined and, if necessary, strengthened, since strict monotonicity of the period functions is required.
\end{remark}

\begin{remark}
	We adopt the sufficient condition obtained by Chow and Wang because all the other sufficient conditions discussed in Ref.~\cite{Chouikha1999} imply it. Nevertheless, as noted therein, this condition does not constitute an optimal criterion for monotonicity.
\end{remark}

\section{Applications}\label{application}

\subsection{A toy example}
Consider the following system in polar coordinates:
\begin{equation}\label{eg1_eq3}
	\begin{cases}
		\dot{r}=0,\\
		\dot{\theta}=\frac{1}{1+r}.
	\end{cases}
\end{equation}

By translating its center to \((1,0)\), we obtain the following system in  Cartesian coordinates \((x,y)\):
\begin{equation}\label{eg1_eq1}
	\begin{cases}
		\dot{x}=-\frac{y}{1+\sqrt{(x-1)^2+y^2}}\\
		\dot{y}=\frac{x-1}{1+\sqrt{(x-1)^2+y^2}}.
	\end{cases}
\end{equation}
Similarly, translating the center of the system
\begin{equation}\label{eg1_eq4}
	\begin{cases}
		\dot{r}=0,\\
		\dot{\theta}=-\frac{1}{1+r},
	\end{cases}
\end{equation}
to $(-1,0)$, we obtain
\begin{equation}\label{eg1_eq2}
	\begin{cases}
		\dot{x}=\frac{y}{1+\sqrt{(x+1)^2+y^2}},\\
		\dot{y}=-\frac{x+1}{1+\sqrt{(x+1)^2+y^2}}.
	\end{cases}
\end{equation}

For convenience, we define the curves \(\gamma_i\) and \(\Gamma_i\) as follows (see \cref{eg1_1}):
\begin{align}
	\gamma_1&:=\big\{\phi_1\big(t;(0,0)\big):t\geq 0\big\},\\
	\Gamma_1&:=\big\{\phi_1\big(t;(-1,0)\big):t\geq 0\big\},\\
	\gamma_2&:=\big\{\phi_2\big(t;(0,0)\big):t\geq 0\big\},\\
	\Gamma_2&:=\big\{\phi_2\big(t;(1,0)\big):t\geq 0\big\},
\end{align}
where \(\phi_1\big(t;(x,y)\big)\) and \(\phi_2\big(t;(x,y)\big)\) denote the flow maps associated with \eqref{eg1_eq1} and \eqref{eg1_eq2}, respectively.
\begin{figure}
	\centering
	\includegraphics[width=\textwidth]{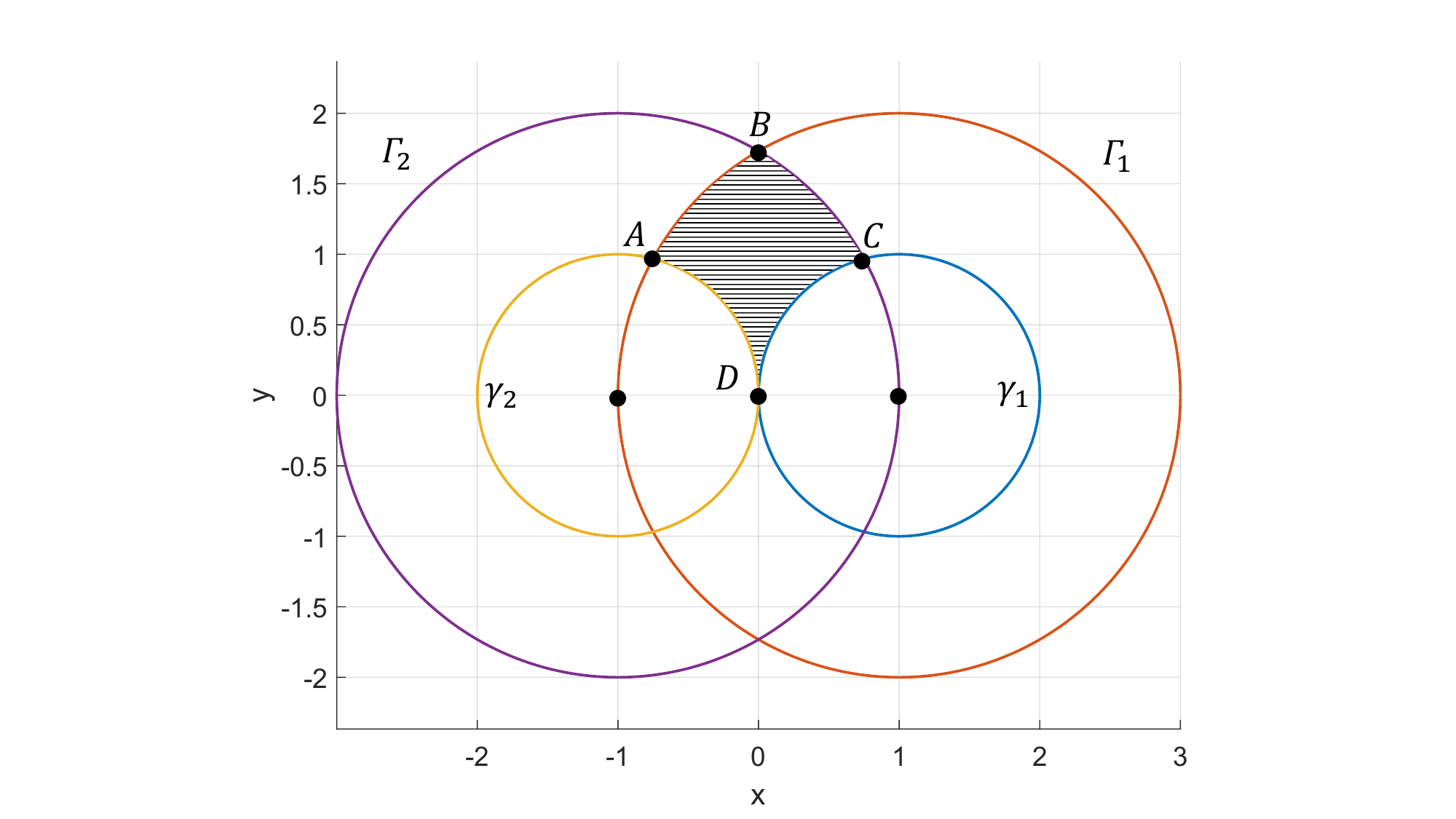}
	\caption{$\gamma_i,\Gamma_i~(i=1,2)$ and the curvilinear quadrilateral $ABCD$.}
	\label{eg1_1}
\end{figure}

From \eqref{eg1_eq3} and \eqref{eg1_eq4}, it follows that
\begin{align}
	\mathcal{T}_1(0,0)&=\mathcal{T}_2(0,0)=4\pi,\\
	\mathcal{T}_1(-1,0)&=\mathcal{T}_2(1,0)=6\pi.
\end{align}
Consequently,
\begin{equation}
	T_1^*=\frac{4\pi\times6\pi}{|4\pi-6\pi|}=12\pi=T_2^*.
\end{equation}
According to \cref{MainTheorem1}, we choose
\begin{equation}
	T_1=5T_1^*=60\pi,\qquad T_2=3T_2^*=36\pi.
\end{equation}
The orbit with initial condition \((0,1)\) over the time interval \([0,1000\pi]\) is depicted in \cref{eg1_2}, while \cref{eg1_3} illustrates the time evolution of \(x(t)\) and \(y(t)\).
The numerical solution is computed using the solver \texttt{ode45} with tolerance settings
\begin{equation}
	\texttt{RelTol}=10^{-8},\qquad \texttt{AbsTol}=10^{-10}.
\end{equation} 

\begin{figure}
	\centering
	\includegraphics[width=0.8\textwidth]{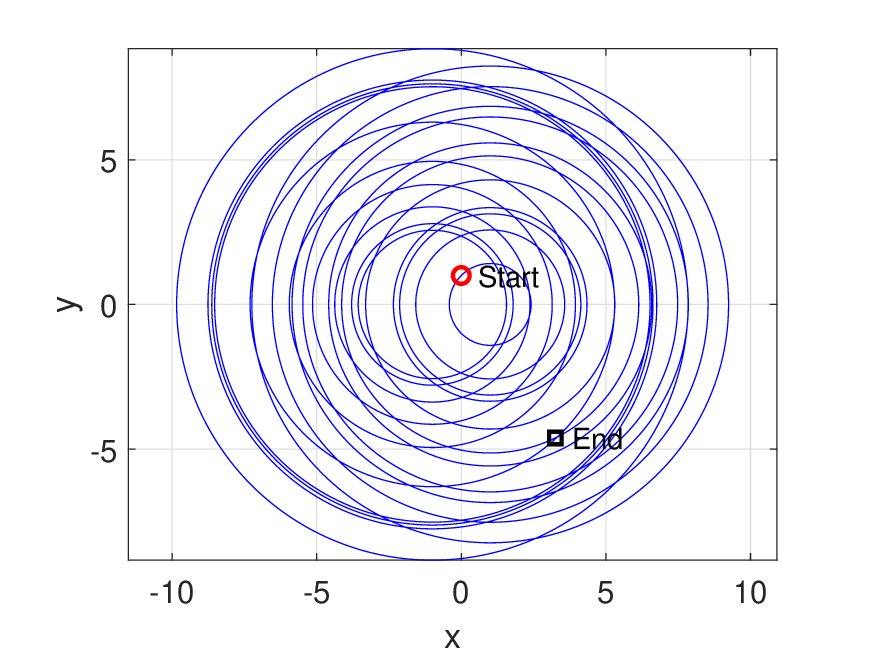}
	\caption{Orbit with initial condition $(0,1)$ over the time interval \([0,1000\pi]\). The circle marks the initial point, and the square marks the terminal point.}
	\label{eg1_2}
\end{figure}
\begin{figure}
	\centering
	\includegraphics[width=0.8\textwidth]{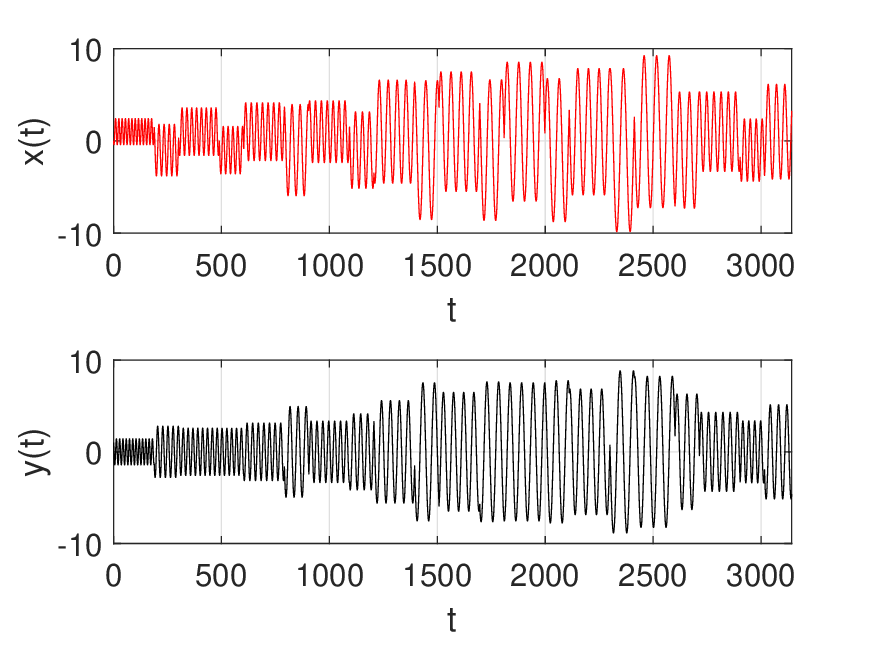}
	\caption{Time evolution of the state variables.}
	\label{eg1_3}
\end{figure}

\subsection{The reduced seasonally forced SIR epidemic model}\label{SIRe}
In Ref.~\cite{Barrientos2017}, the authors reduce the seasonally forced SIR epidemic model to a periodic switching system composed of two Newtonian Hamiltonian subsystems of the form \eqref{SIR}, in order to study chaotic dynamics of the original model.
\begin{equation}\label{SIR}
	\begin{cases}
		\dot{x}=-y,\\
		\dot{y}=B(1+y)(Ax+\alpha_i),
	\end{cases}
\end{equation}
where $\alpha_1=-1,~\alpha_2=1,~A,B>0$.

In this subsection, we derive sharper lower bounds on the dwell time of each subsystem that ensure the occurrence of chaotic dynamics in the switching system. Moreover, in contrast to the symmetric bounds obtained in Ref.~\cite{Barrientos2017}, we show that these bounds need not be identical.

By a slight modification of \cref{MainTheorem3}, or alternatively by applying the classical results in Ref.~\cite{Chow1986}, it follows that the derivative of the period function associated with each energy level \(c>0\) is strictly positive.

For the purpose of numerical simulation, we fix \(A=B=1\), yielding the two subsystems
\begin{equation}\label{eg2_eq1}
	\begin{cases}
		\dot{x}=-y,\\
		\dot{y}=(1+y)(x+\alpha_i).
	\end{cases}
\end{equation}

We define the curves \(\gamma_i\) and \(\Gamma_i\) as follows (see \cref{eg2_1}):
\begin{align}
	\gamma_1&:=\big\{\phi_1\big(t;(-\tfrac13,0)\big): t\geq 0\big\},\\
	\Gamma_1&:=\big\{\phi_1\big(t;(-\tfrac23,0)\big): t\geq 0\big\},\\
	\gamma_2&:=\big\{\phi_2\big(t;(\tfrac13,0)\big): t\geq 0\big\},\\
	\Gamma_2&:=\big\{\phi_2\big(t;(\tfrac23,0)\big): t\geq 0\big\},
\end{align}
where $\phi_i\big(t;(x,y)\big)$ denote the flow map associated with \eqref{eg2_eq1}.
\begin{figure}
	\centering
	\includegraphics[width=0.8\textwidth]{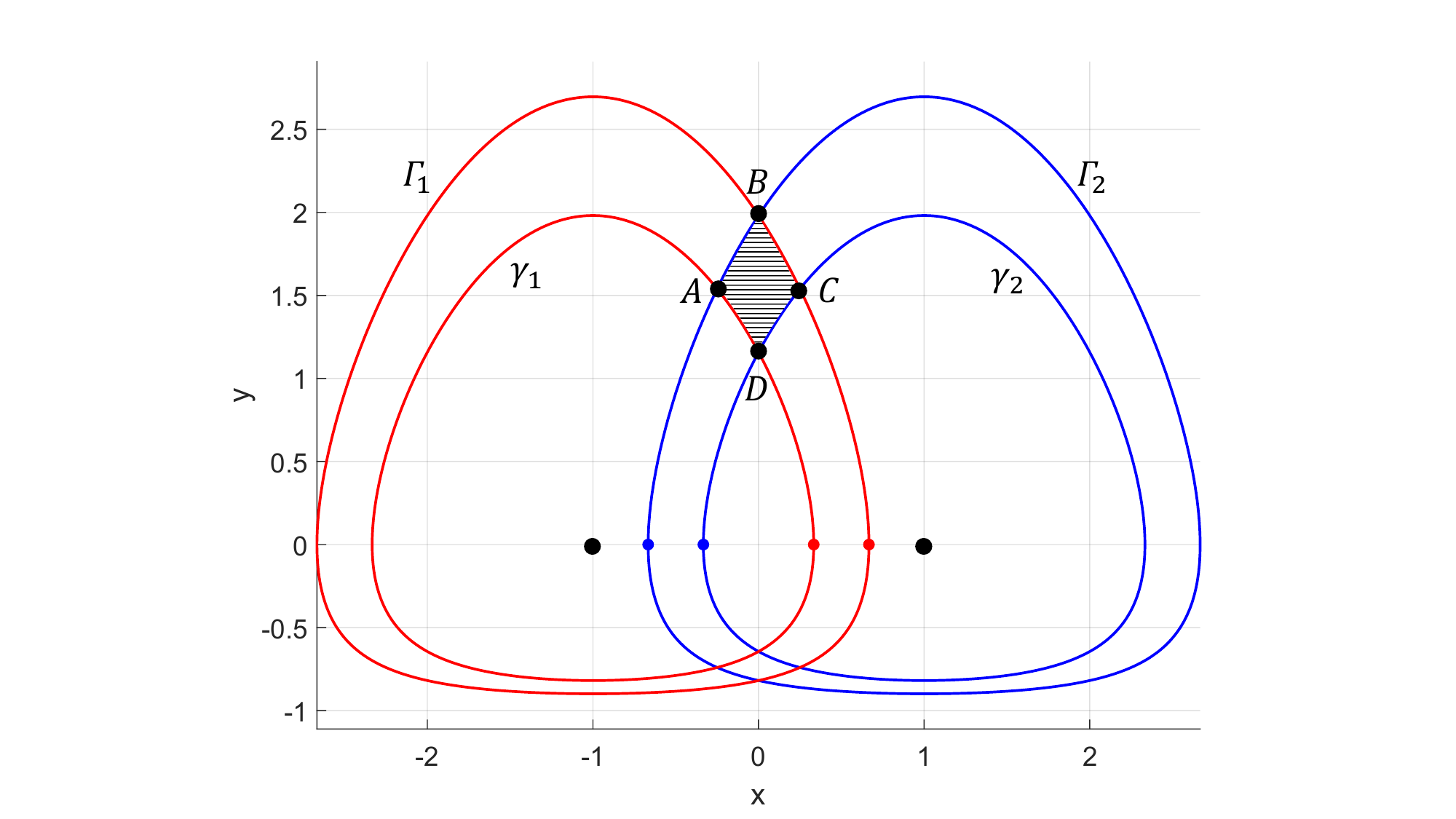}
	\caption{$\gamma_i,\Gamma_i~(i=1,2)$ and the curvilinear quadrilateral $ABCD$.}
	\label{eg2_1}
\end{figure}

Since the period functions of \eqref{eg2_eq1} do not have explicit expressions, we employ a Runge–Kutta scheme to approximate the periods corresponding to \(\gamma_i\) and \(\Gamma_i\) (\(i=1,2\)), obtaining
\begin{align}
	\mathcal{T}_1\big(-\tfrac13,0\big)&=\mathcal{T}_2\big(\tfrac13,0\big)=6.755~238~541~9,\\
	\mathcal{T}_1\big(-\tfrac23,0\big)&=\mathcal{T}_2\big(\tfrac23,0\big)=7.023~911~051~3.
\end{align}
Consequently,
\begin{equation}
	T_1^*=\frac{\mathcal{T}_1(-\tfrac13,0),\mathcal{T}_1(-\tfrac23,0)}{\left|\mathcal{T}_1(-\tfrac13,0)-\mathcal{T}_1(-\tfrac23,0)\right|}=176.602~342~958~647=T_2^*.
\end{equation}

In Ref.~\cite{Barrientos2017}, it is shown that the switching system \eqref{SIR} exhibits chaotic dynamics provided that
\begin{align}
	T_1&>\frac{6\mathcal{T}_1(-\tfrac13,0)\mathcal{T}_1(-\tfrac23,0)}{\left|\mathcal{T}_1(-\tfrac13,0)-\mathcal{T}_1(-\tfrac23,0)\right|}=6T_1^*,\\
	T_2&>\frac{6\mathcal{T}_2(\tfrac13,0)\mathcal{T}_2(\tfrac23,0)}{\left|\mathcal{T}_2(\tfrac13,0)-\mathcal{T}_2(\tfrac23,0)\right|}=6T_2^*.
\end{align}
In contrast, \cref{MainTheorem3} implies that chaotic dynamics already arise under the weaker condition that there exist indices \(i\neq j\) such that \(T_i\geq 5T_i^*\) and \(T_j\geq 3T_j^*\). For example, choosing
\begin{equation}
	T_1=5T_1^*,~T_2=3T_2^*,
\end{equation}
we depict in \cref{eg2_2} the orbit over the time interval \([0,80T_1^*]\) with initial condition \((0,1.5)\).
\begin{figure}
	\centering
	\includegraphics[width=0.8\textwidth]{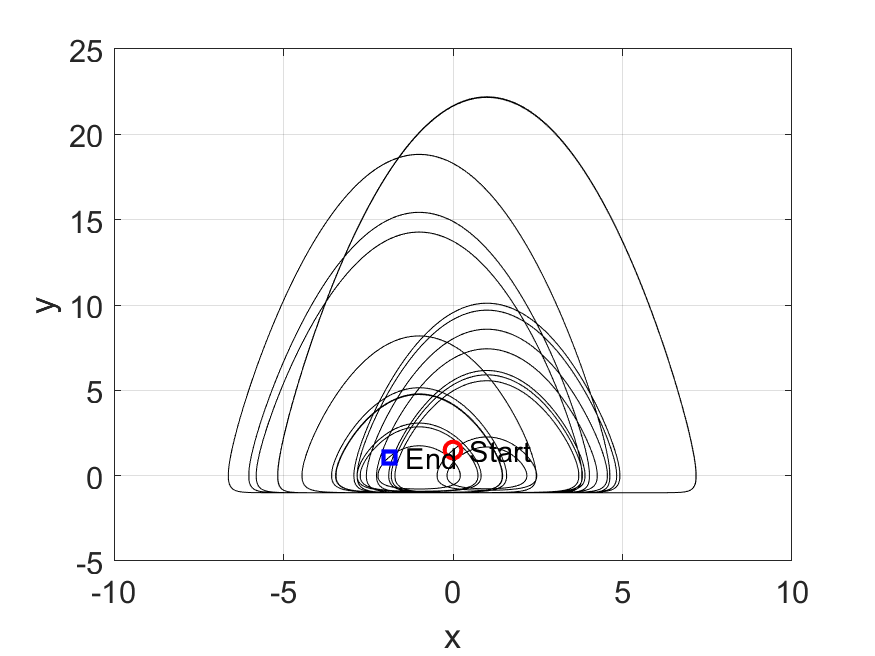}
	\caption{Orbit with initial condition $(0,1.5)$ over the time interval \([0,80T_1^*]\). The circle marks the initial point, and the square marks the terminal point.}
	\label{eg2_2}
\end{figure}
\begin{remark}
	Since the period of the switching system is significantly larger than those of the individual subsystems, higher numerical accuracy is required. Therefore, instead of using \texttt{ode45} as in the previous example, we employ the solver \texttt{ode113} with tolerance settings
	\begin{equation}
		\texttt{RelTol}=10^{-12},\qquad \texttt{AbsTol}=10^{-14},
	\end{equation} 
	to reduce the accumulated numerical error.
\end{remark}

\section{Discussion}\label{discussion}

In contrast to Ref.~\cite{Barrientos2017}, we investigate a broader class of  systems and establish explicit geometric criteria for the existence of the topological horseshoe. 
Moreover, by a more detailed analysis of the mechanism underlying the formation of the topological horseshoe, we derive sharper lower bounds on the dwell times of the subsystems \eqref{SIRe}. Notably, these bounds are shown to be asymmetric, which is somewhat counterintuitive.

Under the assumptions of \cref{MainTheorem1}, the associated Poincaré map $P_2\circ P_1$ (or $P_1\circ P_2$) has a complete topological horseshoe consisting of two crossing blocks. Even when these conditions fail, the system may still exhibit a semi-horseshoe structure in the sense of Ref.~\cite{Huang2021} (see \cref{illustration9}), which likewise guarantees chaotic behavior as shown in Ref.~\cite{Cheng2025a}.

On the other hand, the configuration illustrated in \cref{illustration10} may also occur. This represents a typical non-chaotic scenario, as analyzed in Ref.~\cite{Cheng2025a}. Furthermore, motivated by the notion of semi-horseshoes,  higher iterates of the Poincaré map may further reduce the lower bounds on the dwell times, which remains an open problem.

\begin{figure}
	\centering
	\includegraphics[width=0.8\textwidth]{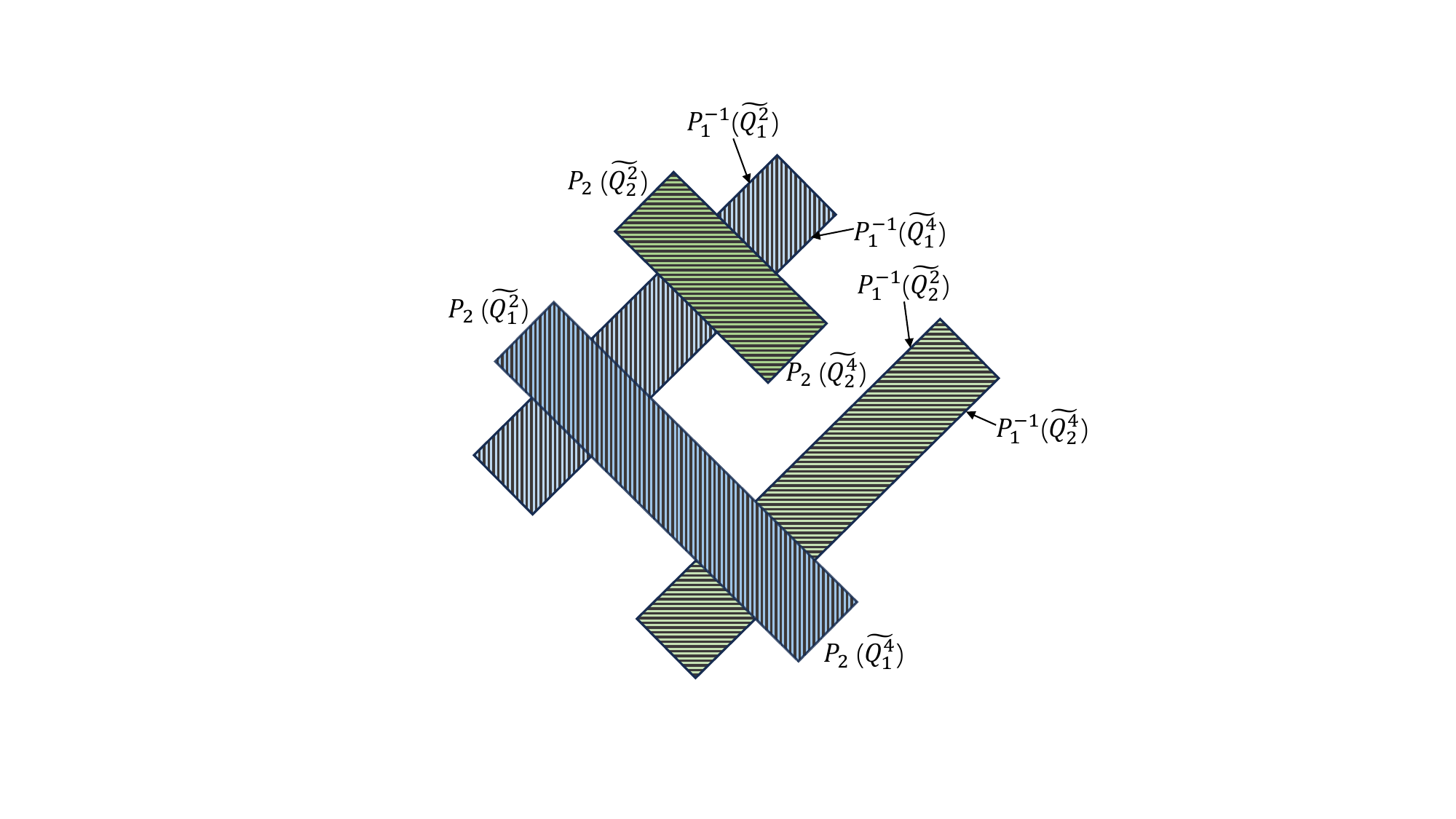}
	\caption{Illustration of a semi-horseshoe that may occur.}
	\label{illustration9}
\end{figure}

\begin{figure}
	\centering
	\includegraphics[width=0.8\textwidth]{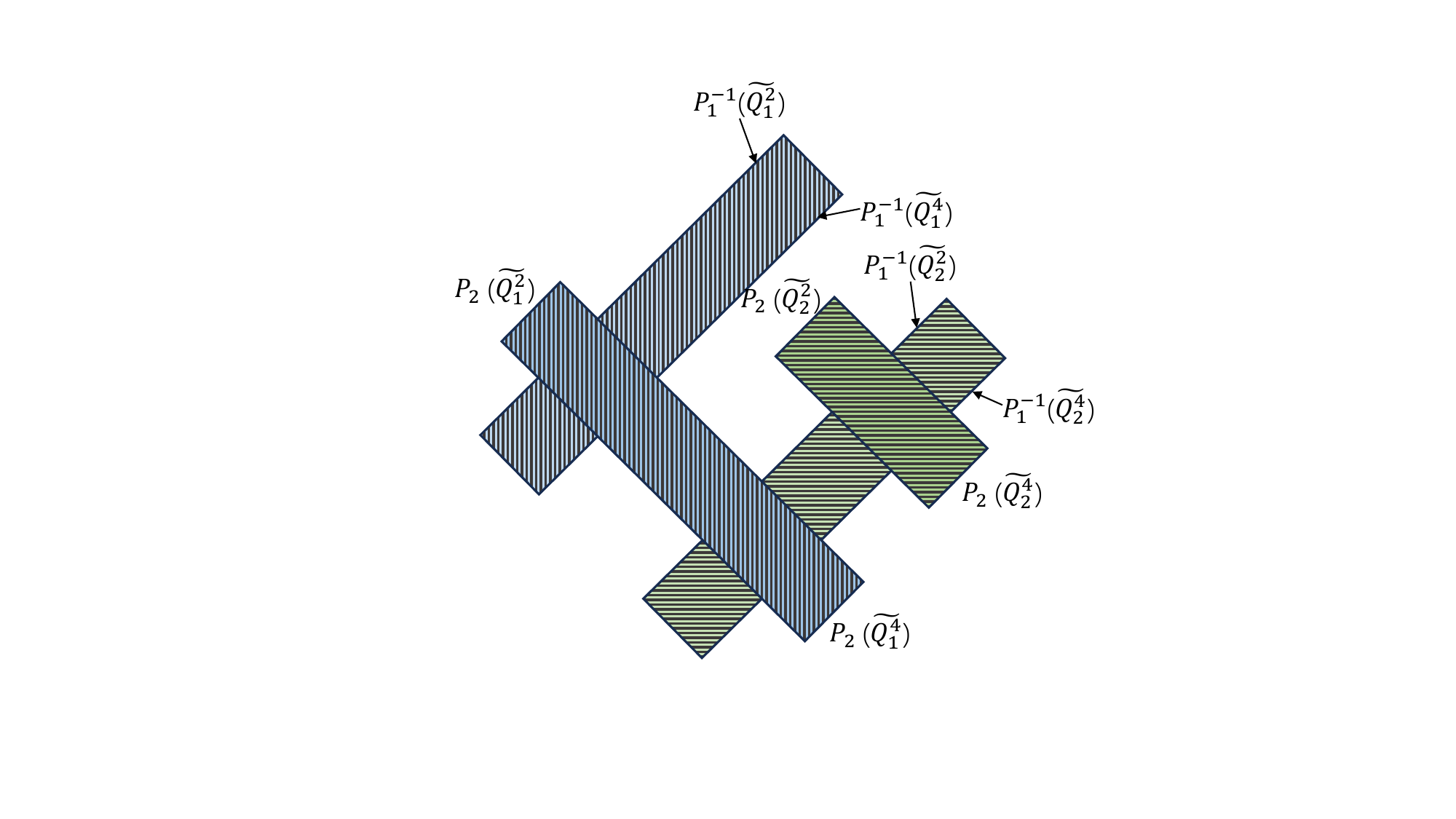}
	\caption{An incomplete horseshoe that does not induce chaotic behavior.}
	\label{illustration10}
\end{figure}

It is also of interest to consider the case where the dwell times $T_1$ and $T_2$ are not fixed. In particular, one can show that if \(T_i \geq 5T_i^*\) and \(T_j \geq 3T_j^*\) for \(i \neq j\), then the required crossing relations are satisfied. However, this does not guarantee the existence of a topological horseshoe, since the corresponding Poincar\'e map is no longer fixed. This phenomenon is of practical relevance and calls for further investigation.

\section{Sumamary}\label{summary}
In this paper, we establish sufficient conditions for the existence of topological horseshoes in general periodic switching systems and periodic switching Hamiltonian systems. In particular, for Newtonian Hamiltonian systems, we derive an explicit and computable criterion based on the monotonicity of the period function, following the approach of Chow and Wang. A toy example, together with refined results for the seasonally forced SIR model, indicates that the lower bounds on the dwell times of the subsystems need not be symmetric.

Several questions remain open. First, it is natural to ask whether sharper lower bounds on the dwell times can be obtained by considering higher iterates of the Poincar\'e map. Second, from a practical perspective, it is important to determine whether the chaotic dynamics persist when the dwell times satisfy the proposed conditions but are allowed to vary in time rather than remain fixed.

\bibliographystyle{unsrt}  % 或其他样式如 unsrt, abbrv, ieeetr 等	
\bibliography{Switching}% Produces the bibliography via BibTeX.

\end{document}